\documentclass[11pt]{article} 
\usepackage{amsmath,amssymb}
\usepackage{amsfonts}
\usepackage{enumerate}
\usepackage{amsthm}      
\usepackage[dvips]{graphicx} 
\usepackage{xcolor}
\usepackage{here}
\numberwithin{equation}{section}
\newtheorem{thm}{Theorem}[section]
\newtheorem{prop}{Proposition}[section]

\newtheorem{cor}{Corollary}[section]
\theoremstyle{definition}\newtheorem{definition}{Definition}[section]
\theoremstyle{definition}
\newtheorem{conj}{Conjecture}[section]

\newtheorem{rem}{Remark}[section]
\newtheorem{ex}{Example}[section]
\def \1{{\mathbf 1}}

\def \f{\frac}

\def \p{\mathfrak p}

\def \Re{\operatorname{Re}}

\def \l{\left}
\def \r{\right}
\def \C{{\mathbb C}}
\def \R{{\mathbb R}}
\def \F{{\mathbb F}}

\def \Q{{\mathbb Q}}
\def \Z{{\mathbb Z}}
\makeatletter
\def\centering{%
   \expandafter\let\csname\@backslashchar\space\endcsname\@centercr
   \rightskip\@flushglue\leftskip\@flushglue
   \parindent\z@
   \parfillskip\z@skip}
\def\raggedright{%
   \expandafter\let\csname\@backslashchar\space\endcsname\@centercr
   \@rightskip\@flushglue \rightskip\@rightskip
   \leftskip\z@skip
   \parindent\z@}
\def\raggedleft{%
   \expandafter\let\csname\@backslashchar\space\endcsname\@centercr
   \rightskip\z@skip \leftskip\@flushglue
   \parindent\z@
   \parfillskip\z@skip}
\makeatother

\title{Chebyshev's Bias against Splitting and Principal Primes in Global Fields}
\author{Miho Aoki\thanks{ 
Department of Mathematics,
Interdisciplinary Faculty of Science and Engineering, Shimane University,
1060, Nishikawatsu, Matsue, Shimane, 690-8504, Japan \\
Email address: aoki@riko.shimane-u.ac.jp \\
ORCID iD: 0000-0001-5203-1763
 } and Shin-ya Koyama\thanks{
 Department of Biomedical Engineering, Toyo University, 2100 Kujirai, 
Kawagoe, Saitama, 350-8585, Japan \\
Email address: koyama@toyo.jp\\
ORCID iD: 0000-0002-8529-5244
 }   
\\ \\       
{ Dedicated to Professor Nobushige Kurokawa on his 70th birthday} 
 } 
\date{ }

\begin{document}
\maketitle
\begin{abstract}
Reasons for the emergence of Chebyshev's bias were investigated.
The Deep Riemann Hypothesis (DRH) enables us to
reveal that the bias is a natural phenomenon for achieving a well-balanced disposition of the whole sequence of primes,
in the sense that the Euler product converges at the center.
By means of a weighted counting function of primes, 
the authors succeed in expressing magnitudes of the deflection by a certain asymptotic formula
under the assumption of DRH,
which provides a new formulation of Chebyshev's bias.

For any Galois extension of global fields and for any element $\sigma$ in the Galois group, we have established
a criterion of the bias of primes whose Frobenius elements are equal to $\sigma$ under the assumption of DRH.
As an application we have obtained a bias toward non-splitting and non-principle primes in abelian extensions
under DRH. 
In positive characteristic cases, DRH is known, and all these results hold unconditionally.
\end{abstract}
\noindent
{\bf Mathematics Subject Classification}\quad  11A41, 11M41, 11N05, 11N13, 11R20, 11R37, 11R42, 11R44,  11R59

\section{Introduction}
Chebyshev's bias is the phenomenon that
there tend to be more primes of the form $4k+3$ than of the form $4k+1$ $(k\in\Z)$.
In fact, if denoting by $\pi(x;\,q,\,a)$ the number of primes $p\le x$ such that $p\equiv a\pmod q$,
then the inequality 
\begin{equation}\label{1}
\pi(x;\, 4,\, 3)\ge \pi(x;\, 4,\, 1)
\end{equation}
holds for any $x$ less than $26861$, 
which is the first prime number that violates the inequality \eqref{1}.
However, both sides draw equal at the next prime 26863, and $\pi(x;\, 4,\, 3)$ moves ahead again until
$616841$.
It is computed that more than 97\% of $x<10^{11}$ satisfy the inequality \eqref{1}.
In spite of that, Littlewood \cite{Li} proved that the difference
$\pi(x;\, 4,\, 3)-\pi(x;\, 4,\, 1)$ changes its sign infinitely many times. 
Moreover, Knapowski and Turan (1962) conjectured that the limit of the percentage in all positive numbers of the set
$$
A_X=\{x<X\ |\ \pi(x;\, 4,\, 3) \ge \pi(x;\, 4,\, 1)\}
$$
as $X\to\infty$ would be equal to 100\%.
However, it is proved now under GRH
that the limit does not exist and that the conjecture is false \cite{Kac}.
Fiorilli and Jouve \cite{FJ} have recently proved that there exist  infinitely many Galois extensions $L/K$ 
of number fields and associated conjugacy classes
$c_{\sigma}, c_{\tau} \, (\sigma ,\, \tau \in \mathrm{Gal}(L/K))$  of the same size such
that  the set
$$
A_X=\{x<X\ |\ \pi(x;\, \sigma)/ |c_{\sigma}| \geq \pi(x;\, \tau)|c_{\tau}|\}
$$
has natural density asymptotically equal to $1$, where $ \pi(x;\, \sigma)$ (resp.\,$ \pi(x;\, \tau)$) is 
the number of prime ideals of norm up to $x$ with Frobenius conjugacy class $c_{\sigma}$
(resp.\,$c_{\tau}$).

In place of such a naive density, the logarithmic density plays the role.
Define the \textit{logarithmic density} of the set $A_X$ in $[2,X]$ by
\begin{equation}\label{ld}
\delta(A_X)=\f1{\log X}\int_{t\in A_X}\f{dt}t.
\end{equation}
Rubinstein-Sarnak \cite{RS} proved that $\lim_{X\to\infty}\delta(A_X)=0.9959...$
under the assumption of the GRH and the 
GSH (Grand Simplicity Hypothesis) for the $L$-functions.

It is known by Dirichlet's prime number theorem in arithmetic progressions 
that the numbers of primes of the form $4k + 3$ and $4k + 1$ are to asymptotically equal. 
Then we have found from the following insight that
Chebyshev's bias means that the primes of the form $4k +3$ appear earlier than those of the form $4k + 1$.

For example, in the case that among the first 100 prime numbers, 50 of the first half are of the form $4k +3$
and those of the latter half are of the form $4k + 1$,
the inequality $\pi(x;\, 4,\, 3) \ge \pi(x;\, 4,\, 1)$ always holds in this interval 
even if their total number of elements are equal, and the maximum difference becomes 50. 
On the other hand, in the case that primes of the form $4k +3$ and $4k + 1$ appear alternately,
the maximum difference may be 1 even if the same inequality holds incessantly.
From this discussion, it is effective to apply the structure that 
``regards smaller primes as heavier elements", 
reflecting the magnitudes of the primes, so as to elucidate Chebyshev's bias. 
One of the reasons why the logarithmic density was effective would be that it treated the contribution of 
smaller numbers as greater ones by virtue of the factor $1/t$ in the integral in \eqref{ld}.

Therefore in this paper, we utilize the following weighted counting function
$$
\pi_s(x;\,q,\,a)=
\sum_{\genfrac{}{}{0pt}{1}{p<x:\,\text{prime}}{p\equiv a\pmod q}}\f{1}{p^s}\qquad(s\ge0).
$$
This is a generalization of the counting function $\pi(x;\,q,\,a)=\pi_0(x;\,q,\,a)$.
Here, small primes allow higher contribution to $\pi_s(x;\,q,\,a)$, as long as we fix $s>0$.
The function $\pi_s(x;\,q,\,a)$ $(s>0)$ should be more appropriate to represent the phenomenon,
as it reflects the size of primes which $\pi(x;\,q,\,a)$ has ignored.

Indeed, although the natural density of the set
$$
A(s)=\{x>0\ |\ \pi_s(x;\,4,\,3)-\pi_s(x;\,4,\,1)>0\}
$$
did not exist for $s=0$ under GRH, it is revealed in this paper under the assumption of
the Deep Riemann Hypothesis (DRH) that
it would exist and equal to 1 when $s=1/2$, that is,
$$
\lim_{X\to\infty}\f1X\int_{t\in A(1/2)\cap[2,X]}dt=1.
$$
More precisely we would reach under DRH that
\begin{equation}\label{cheby}
\pi_{\f12}(x;\,4,\,3)-\pi_{\f12}(x;\,4,\,1)
\sim\f12\log\log x\quad(x\to\infty),
\end{equation}
where $f(x)\sim g(x)$ $(x\to\infty)$ means that
$f(x)/g(x)\to1$ as $x\to\infty$.
The asymptotic \eqref{cheby} suggests a formulation of Chebyshev's bias.

This observation is derived from the conditional convergence of the 
Euler products of Dirichlet $L$-functions, which was first notice by Goldfeld \cite{G} in 1982,
and later named the Deep Riemann Hypothesis (DRH) by Kurokawa \cite{Ku}.
In the case of Dirichlet $L$-functions $L(s,\chi)$
for non-principal Dirichlet characters $\chi$, DRH states that
it holds on $\Re(s)=1/2$ that
\begin{multline}\label{DRH1}
\lim_{x \to \infty} \Bigg((\log x)^{m} \prod_{p \le x:\,\text{prime}} 
\l(1-\f{\chi(p)}{p^s}\r)^{-1} \Bigg)\\
 = \frac{L^{(m)}(s, \chi)}{e^{m \gamma} m!}
\times
\begin{cases}
\sqrt2 & (\chi^2=1,\ s=\frac12)\\
1 & (\text{otherwise})
\end{cases}
\end{multline}
with $\gamma$ the Euler constant and $m=m_\chi$ is the order of zero of $L(s,\chi)$ at $s$.
Taking the logarithm of \eqref{DRH1} for the case $\chi^2=1$ and $s=\f12$, the convergence of 
the above limit provides the following:
$$
\sum_{p\le x:\,\text{prime}}\f{\chi(p)}{\sqrt p}+\l(\f12+m\r)\log\log x=O(1)\quad(x\to\infty).
$$
In particular, if $q$ is an odd prime and $\chi$: $(\Z/q\Z)^\times\to\C^\times$ is defined by
$$
\chi(a)=\l(\f aq\r)
=\begin{cases}
1 & (a\in ((\Z/q\Z)^\times)^2)\\
-1& (\text{otherwise}),
\end{cases}
$$
then we would have
\begin{align*}
&
\sum_{\genfrac{}{}{0pt}{1}{b\in (\Z/q\Z)^\times}{\text{quadratic non-residue}}}
\pi_{\f12}(x;\,q,\,b)-
\sum_{\genfrac{}{}{0pt}{1}{a\in (\Z/q\Z)^\times}{\text{quadratic residue}}}
\pi_{\f12}(x;\,q,\,a)\\
&=
\sum_{\genfrac{}{}{0pt}{1}{p\le x:\,\text{prime}}{\text{quadratic non-residue}}}\f{1}{\sqrt p}
-\sum_{\genfrac{}{}{0pt}{1}{p\le x:\,\text{prime}}{\text{quadratic residue}}}\f{1}{\sqrt p}\\
&\sim
\l(\f12+m\r)\log\log x\quad(x\to\infty).
\end{align*}
A more precise result on the difference between individual
$\pi_{\f12}(x;\,q,\,b)$ and $\pi_{\f12}(x;\,q,\,a)$ is proved in \eqref{eq:difference} in \S3 under DRH.
Chebyshev's original case \eqref{cheby} is restored by choosing $(q,\, a,\, b)=(4,\, 1,\, 3)$ in Example \ref{original}.

\bigskip
Throughout this paper $K$ is supposed to be a global field.
Virtue of the formula \eqref{cheby} allows to reach a formulation of the Chebyshev bias of
primes $\p$ of $K$ as follows.
\begin{definition}
Let $a(\p)\in\R$ be a sequence over prime ideals $\p$ of $K$ such that
$$
\lim_{x\to\infty}\f{\#\{\mathfrak p\ |\ a(\p)>0,\ N(\p)\le x\}}
{\#\{\mathfrak p\ |\ a(\p)<0,\ N(\p)\le x\}}
=1.
$$
The sequence $a(\p)$ has a \textit{Chebyshev bias to being positive}, if
there exists $C>0$ such that
$$
\sum_{N(\p)\le x}\f{a(\p)}{\sqrt{N(\p)}}\sim C\log\log x\quad(x\to\infty),
$$
where $\p$ runs through primes of $K$.
On the other hand, $a(\p)$ is \textit{unbiased}, if
$$
\sum_{N(\p)\le x}\f{a(\p)}{\sqrt{N(\p)}}=O(1)\quad(x\to\infty).
$$
\end{definition}

\begin{definition}
Assume that the set of all primes $\p$ of $K$ with $N(\p)\le x$ 
is expressed as a disjoint union $P_1(x)\cup P_2(x)$ and that their proportion converges to
$$
\delta=\lim_{x\to\infty}\f{|P_1(x)|}{|P_2(x)|}.
$$
There exists a \textit{Chebyshev bias toward} $P_1$
(or \textit{Chebyshev bias against} $P_2$),  if the following asymptotic holds:
$$
\sum_{p\in P_1(x)}\f1{\sqrt{N(\p)}}-
\delta \sum_{p\in P_2(x)}\f1{\sqrt{N(\p)}}
\sim C\log\log x\quad(x\to\infty)
$$
for some $C>0$.
On the other hand, there exist \textit{no biases} between $P_1$ and $P_2$, if the following holds:
$$
\sum_{p\in P_1(x)}\f1{\sqrt{N(\p)}}-
\delta \sum_{p\in P_2(x)}\f1{\sqrt{N(\p)}}
=O(1)\quad(x\to\infty).
$$
\end{definition}

\bigskip
Denote an $n$-dimensional nontrivial irreducible 
Artin representation by
$$\rho:\ \mathrm{Gal}(K^{\mathrm{sep}}/K)\to \mathrm{Aut}_\C(V)\quad (\rho\ne\mathbf 1).$$
The $L$-function $L_K(s,\rho)$ is defined by the Euler product as
$$
L_K(s,\rho)=\prod_{\mathfrak p} 
\det\l(1-\rho(\mathrm{Frob}_{\mathfrak p} )N(\mathfrak p)^{-s} |V^{I_{\mathfrak p}}\r)^{-1}
\quad (\Re(s)>1),
$$
where $\mathfrak p$ runs through the prime ideals of $K$ with $N(\mathfrak p)$ its norm
and $\mathrm{Frob}_{\mathfrak p}\in\mathrm{Gal}(K^{\mathrm{sep}}/K)$ the Frobenius element
with $I_{\mathfrak p}$ the inertia group.

\begin{conj}[Deep Riemann Hypothesis (DRH)]\label{DRH}
Put $m=m_\rho=\mathrm{ord}_{s=\f12}L_K(s,\rho)$.
Then the limit 
\begin{multline}\label{limit}
\lim_{x \to \infty} \Bigg((\log x)^{m} \prod_{N(\mathfrak p) \le x} 
\det\l(1-\rho(\mathrm{Frob}_{\mathfrak p})N(\mathfrak p)^{-\f12} | V^{I_{\mathfrak p}}\r)^{-1}  \Bigg)
\end{multline}
satisfies the following:
\begin{description}
\item[DRH(A)] The limit \eqref{limit} exists and is nonzero.
\item[DRH(B)] The limit \eqref{limit} satisfies the following identity:
\begin{multline*}
\lim_{x \to \infty} \Bigg((\log x)^{m} \prod_{N(\mathfrak p) \le x} 
\det\l(1-\rho(\mathrm{Frob}_{\mathfrak p})N(\mathfrak p)^{-\f12} | V^{I_{\mathfrak p}}\r)^{-1} \Bigg)\\
 = \frac{\sqrt{2}^{\nu(\rho)} L_{K}^{(m)}(1/2, \rho)}{e^{m \gamma} m!},
\end{multline*}
where
$\nu(\rho) = \mathrm{mult}(\mathbf 1, \mathrm{sym}^{2} \rho)-\mathrm{mult}(\mathbf 1, \wedge^{2} \rho) \in \Z$
with $\mathrm{mult}(\mathbf 1, \sigma)$ 
being the multiplicity of the trivial representation $\mathbf 1$ in $\sigma$.
\end{description}
\end{conj}

\bigskip
The facter $\sqrt 2$ in the right hand side of DRH(B) was discovered by the pioneering work of Goldfeld \cite{G}.

Obviously DRH(B) implies DRH(A). 
Also note that if $m>0$, DRH(A) implies the following:
\begin{equation*}
\lim_{x \to \infty} \prod_{N(\mathfrak p) \le x} 
\det\l(1-\rho(\mathrm{Frob}_{\mathfrak p}) N(\mathfrak p)^{-\f12}|V^{I_{\mathfrak p}}\r)^{-1}
=L_K(1/2,\rho)=0.
\end{equation*}

Justification of DRH may depend on the pair $(K,\rho)$.
In what follows in this paper,
whenever DRH is assumed, the types of $(K,\rho)$ on which $L_K(s,\rho)$ should satisfy DRH are specified.
However, since $K$ is fixed throughout this paper, 
``DRH for $L_K(s,\rho)$'' is abbreviated to ``DRH for $\rho$'' in some cases.

\bigskip
Conjecture \ref{DRH} originates from Birch and Swinnerton-Dyer \cite[p.79 (A)]{BSD}.
They conjectured for the representation $\rho_E$ attached to an elliptic curve $E/\Q$ that
the Euler product of $L_\Q(s,\rho_E)$ should converge at the center,
and that in case the limit is zero, it asymptotically equals $C(\log x)^g$
with $g=\mathrm{rank}(E)$.
In the celebrated work \cite{G},
Goldfeld proved that their conjecture implies the Generalized Riemann Hypothesis (GRH) for $L_\Q(s,\rho_E)$
and that if their conjecture is correct, then $g=m$ holds.

It is also discovered by Conrad \cite[Theorems 3.3 and 6.3]{Co} that DRH implies both 
the convergence of the Euler product in $\Re(s)\ge 1/2$ 
and the GRH for $L_K(s,\rho)$.

When $\mathrm{char}(K)>0$, 
both DRH(A) and (B) are proved. 
The proof is substantially given by Conrad \cite[Theorems 8.1 and 8.2]{Co} in a different context
under the assumption of the second moment hypothesis,
and the full proof is recently given by Kaneko-Koyama-Kurokawa
\cite{KKK2}             
along the formulation of DRH.
The main theorems in this paper provide a criterion of emergence of the Chebyshev biases under the assumption of DRH,
which means the criterion holds unconditionally for $\mathrm{char}(K)>0$.

\bigskip
In this paper we examine Chebyshev biases existing in the primes of global fields.
Let $L/K$ be a finite Galois extension of global fields.
The main theorem is described in terms of the set $S$ of all primes in $K$ and
its subset $S_\sigma$ of unramified primes whose Frobenius element $\left( \frac{L/K}{\mathfrak p} \right)$
is equal to $\sigma\in\mathrm{Gal}(L/K)$:

\begin{thm}[a part of Theorem \ref{thm:main}]
Let $L/K$ be a finite Galois extension of global fields. 
The following (i) and (ii) are equivalent:
\begin{enumerate}[\textrm(i)]                                                                        
\item DRH(A) for all non trivial irreducible representations of Gal$(L/K)$.          
\item 
For any $\sigma \in $Gal$(L/K)$ it holds that
\begin{multline*}
\sum_{\mathfrak p\in S \atop N(\mathfrak p) \leq x } \frac{1}{\sqrt{N(\mathfrak p)}}
-\frac{[L:K]}{|c_{\sigma} | }
\sum_{\mathfrak p \in S_{\sigma} \atop N(\mathfrak p) \leq x} \frac{1}{\sqrt{N(\mathfrak p)}}
=C\log{\log{x}}+c+o(1)\\
(x\to \infty)
\end{multline*}
for some constants $C$ and $c$ depending on $\sigma$.
\end{enumerate}    
\end{thm}

Here $C$ is expressed in terms of $\nu(\rho)$ and $m_\rho$ in Conjecture \ref{DRH}.
Various examples of Chebyshev biases are obtained by calculating such constants for specific cases 
under the assumption of (i).
Some examples are presented at the end of this section.
Note that Chebyshev's bias is a problem beyond the Riemann Hypothesis,
since it is essentially equivalent to DRH(A).

The situation for the case of entire $L$-functions\footnote{
In case the $L$-functions have a pole at $s=1$, DRH needs modifications by Akatsuka \cite{A}.
} is illustrated in Figure 1.
Here we denote the Generalized Riemann Hypothesis (GRH) by just the Riemann Hypothesis (RH),
as our interest is whether the hypothesis is deep or not.
The Riemann Hypothesis is located in the upper left box by this notation, 
which follows from the upper middle box, the convergence of the Euler product in $\f12<\Re(s)<1$, 
since convergence of an infinite product implies by definition that the limit is nonzero.
Extension of the region of convergence to the boundary $\Re(s)=\f12$ is in the middle box, which is DRH(A).
The theme of this paper is the essential equivalence between Chebyshev's bias and DRH(A).
Although it is weaker than the full Deep Riemann Hypothesis (DRH(B)), 
it is still stronger than the Riemann Hypothesis.
For numerical evidence of DRH, see \cite[Table I]{KKK}.
\begin{figure}[H]   
\begin{center}
\includegraphics[scale=0.5]{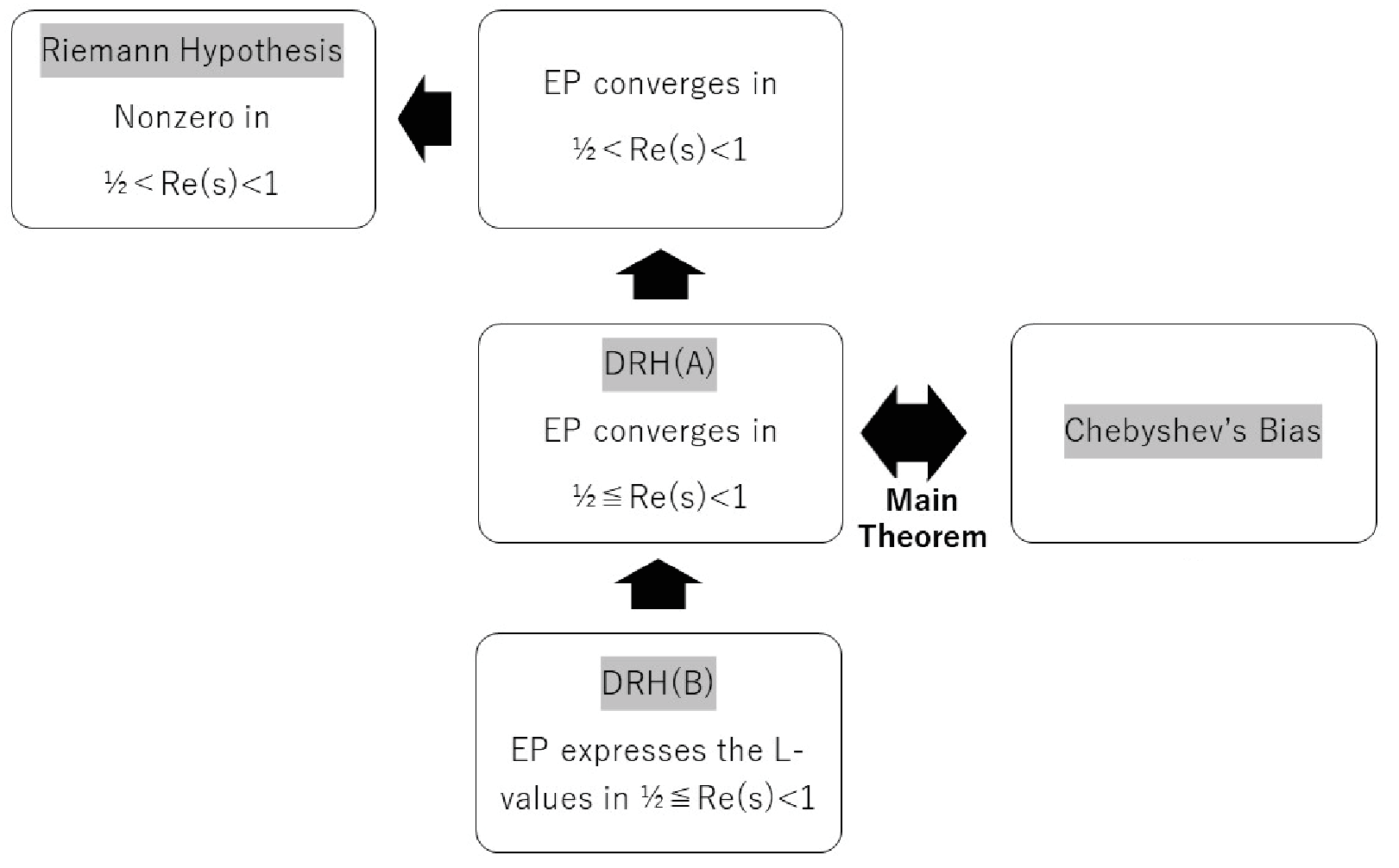}
\caption{Deep Riemann Hypothesis (DRH) and Chebyshev's Bias \newline\hfil (in case of entire $L$-functions)}
\end{center}
\end{figure}

This section is concluded by introducing three simple examples.

\begin{ex}[Bias against splitting primes (Example \ref{ex:decompose})]\label{ex:splitting}
Assume $[L:K]=2$ and let $\chi$ be the nontrivial character of Gal$(L/K)$.
The following (i) and (ii) are equivalent:
\begin{enumerate}[\textrm(i)]
\item DRH(A) for $L_K(s,\chi)$ holds.
\item There exists a Chebyshev bias against splitting primes with the asymptotic
\begin{multline*}
\sum_{\mathfrak p:\text{ nonsplit} \atop N(\mathfrak p) \leq x } \frac{1}{\sqrt{N(\mathfrak p)}}
-\sum_{\mathfrak p:\text{ split} \atop N(\mathfrak p) \leq x} \frac{1}{\sqrt{N(\mathfrak p)}}\\
=\left( \frac{1}{2} +m_{\chi} \right) \log{\log{x}} + c+o(1)
\qquad (x\to\infty)
\end{multline*}
for some constant $c$.
\end{enumerate}
\end{ex}

It is generally observed that real zeros of $L$-functions have an effect on increasing the bias
(For example, see \cite{Fi}).
Example 1.1 elucidates the reason for this phenomenon since the factor $\f12+m_\chi$
in the right hand side of (ii) becomes bigger as $m_\chi$ grows.

\begin{ex}[Bias against quadratic residues (Corollary \ref{cor:quadraticresidue})]\label{ex:residues}
Let $q$ be a positive integer, 
$t (q)$  the number of distinct prime numbers dividing $q$, and 
$$
t:= \begin{cases}
t (q) -1 & (2 ||q), \\
t (q) & (4||q \text{ or } 2\nmid q),\\
t (q)+1 & (8 |q).
\end{cases}
$$
Assume that $L(\f12,\chi)\ne0$ for all Dirichlet characters $\chi$ modulo $q$.
The following (i) and (ii) are equivalent:
\begin{enumerate}[\textrm(i)]
\item DRH(A) holds for $L(s,\chi)$ for any Dirichlet character $\chi$ modulo $q$.
\item There exists a Chebyshev bias against quadratic residues modulo $q$ with the asymptotic
\[
\pi_{\frac{1}{2}} (x;q,b) -\pi_{\frac{1}{2}}(x;q,a)=\frac{2^{t-1}}{\varphi (q) }\log{\log{x}} +c+o(1)
\quad (x\to \infty)
\]
for some constant $c$ and for any pair $(a,b)$ of a quadratic residue $a$ and a non-residue $b$,
and there exist no biases for any pair $(a,b)$ with both $a$ and $b$ quadratic residues (or non-residues).
\end{enumerate}
\end{ex}

\begin{ex}[Bias against principal ideals (Corollary \ref{cor:principal})]\label{ex:principal}
We denote by $\widetilde{K}$ the Hilbert class field of $K$.
The ideal class group is expressed as $\mathrm{Cl}_K \simeq \mathrm{Gal} (\widetilde{K}/K)$.
An ideal class $[\mathfrak a] \in \mathrm{Cl}_K$ corresponds to $\sigma_{\mathfrak a}:=\left( \frac{\widetilde{K}/K}{\mathfrak a} \right) \in \mathrm{Gal}(\widetilde{K}/K)$. 
Assume DRH(A) for $L_K(s,\chi)$ and that $L_K(\f12,\chi)\ne 0$ 
for any character $\chi$ of $\mathrm{Cl}_K$.
If $|\mathrm{Cl}_K|$ is even, then in the whole set of prime ideals of $K$,
there exists a Chebyshev bias against principal ideals. 
Namely,
\begin{multline*}
\sum_{\mathfrak p:\text{ nonprincipal}\atop N(\mathfrak p ) \leq x} \frac{1}{\sqrt{N(\mathfrak p)}} 
-(h_K-1) \sum_{\mathfrak p:\text{ principal} \atop N(\mathfrak p)\leq x } \frac{1}{\sqrt{N(\mathfrak p)}} \\
=\frac{ |\mathrm{Cl}_K/\mathrm{Cl}_K^2| -1}{2} \log{\log{x}}  +c+o(1)\quad (x \to \infty)
\end{multline*}
with $h_K=|\mathrm{Cl}_K|$.
\end{ex}

\section{Bias of primes in global fields}

Let $K$ be a global field, that is, a number field or an algebraic function field of one variable over
a finite field $\mathbb F_q$ ($q$ is a power of a prime number), and $L/K$ a finite 
Galois extension.
For a prime ideal $\mathfrak p$ 
of $K$
which does not divide the relative discriminant $D_{L/K}$,
we denote by $\left( \frac{L/K}{\mathfrak p} \right)$ the Frobenius element of Gal$(L/K)$.
For any $\sigma \in $Gal$(L/K)$, it is known from the Chebotarev density theorem
that the density of prime ideals $\mathfrak p$ of $K$ satisfying $\mathfrak p \nmid D_{L/K}$ and
$\left(\frac{L/K}{\mathfrak p} \right) 
\sim \, 
\sigma$ is $|c_{\sigma} |/[L:K]$, where $c_{\sigma}$ is 
the conjugacy class of $\sigma$
(See \cite[p125, Theorem~9.13~A]{R2} for function fields).
Let $S$ be the set of all prime ideals of $K$ and put
$S_{\sigma} := \left\{ \mathfrak p \in S \ \vline \ \mathfrak p \nmid D_{L/K},\ \left(\frac{L/K}{\mathfrak p} \right)
\sim \, 
\sigma \right\}$
for $\sigma \in $Gal$(L/K)$.
\begin{thm}[Chebotarev density theorem]\label{theo:chebotarev}
For any  finite Galois  extension $L/K$  of global fields and $\sigma \in \mathrm{Gal}(L/K)$, 
we have
\[
\lim_{s\to 1^+} \frac{\sum_{\mathfrak p \in S_{\sigma}} N (\mathfrak p)^{-s} }{\sum_{\mathfrak p\in 
S
}N(\mathfrak p)^{-s}} =
\frac{|c_{\sigma}|}{[L:K]}.
\]
\end{thm}
We will examine biases of prime ideals toward or against the set $S_{\sigma}$ depending on $\sigma$,
under the assumption of DRH(A) in the case of number fields and unconditionally in the case of algebraic function fields.
For $x\in \R_{>0}$, we put
\begin{align*}
\pi_{s,K}(x) & := \sum_{\mathfrak p\in S \atop N(\mathfrak p) \leq x } \frac{1}{N(\mathfrak p)^s},\\
\pi_{s}(x;\sigma) & := \sum_{\mathfrak p \in S_{\sigma} \atop N(\mathfrak p) \leq x} \frac{1}{N(\mathfrak p)^s},
\end{align*}
for any $\sigma \in $Gal$(L/K)$ and
$$
\pi_{s}(x;H) := \sum_{\sigma \in H} \pi_{s}(x;\sigma)
$$
for any $H\subset \mathrm{Gal}(L/K)$.
\begin{thm}\label{thm:main}
Let $L/K$ be a finite Galois extension of global fields. 
For any $\sigma \in $Gal$(L/K)$ we put
$$
M(\sigma ) :=\frac{1}{2} \sum_{\rho \ne {\textbf 1} \atop \text{irred.} } \chi_{\rho} (\sigma) \nu (\rho) 
$$
with $\chi_{\rho}$ being the character of $\rho$ and 
$$
m(\sigma)  := \sum_{\rho \ne {\textbf 1} \atop \text{irred. }} \chi_{\rho} (\sigma )m_{\rho},
$$
where $\rho$ runs through all nontrivial irreducible representations of Gal$(L/K)$.
The following (i),(ii) and (iii) are equivalent:                                     
\begin{enumerate}[\textrm(i)]                                                                        
\item DRH(A) for all non trivial irreducible representations of Gal$(L/K)$.          
\item For all $\sigma \in $Gal$(L/K)$ it holds that                                     
\begin{multline*}
\pi_{\frac{1}{2},K }(x) -\frac{[L:K]}{|c_{\sigma} | } \pi_{\frac{1}{2} }(x;\sigma) =(M(\sigma )+m(\sigma )) \log{\log{x}} +c+o(1)\\
(x\to \infty)
\end{multline*}
for some constant $c$.          
\item For all pairs $\sigma, \, \tau \in $Gal$(L/K)$ it holds that 
\begin{multline*}
\frac{1}{|c_{\tau}|} 
\pi_{\frac{1}{2} } (x;\tau )-
\frac{1}{ |c_{\sigma}|} 
\pi_{\frac{1}{2}} (x;\sigma) \\
=\frac{1}{[L:K]} \left( M(\sigma) -M(\tau) +m(\sigma)-m(\tau )\right) \log{\log{x}} +c+o(1)\\
(x\to \infty)
\end{multline*}
for some constant $c$.  
\end{enumerate} 
\end{thm}
\begin{rem}
For any irreducible representation $\rho$, the complex conjugate representation $\overline{\rho}$ is irreducible and
we obtain $\nu (\rho)=\nu (\overline{\rho}), \ m_{\rho} =m_{\overline{\rho}}$ and that $\chi_{\overline{\rho}} (\sigma )
=\overline{ \chi_{\rho} (\sigma)}$. Therefore $M(\sigma ), m(\sigma) \in \mathbb R$ holds.
\end{rem}
\begin{ex}[Quaternion extension]\label{ex:Quaternion}
Let $L/\mathbb Q$ be a Galois extension with Gal$(L/\mathbb Q)$ isomorphic to the
quaternion group $Q_8 =\{ \pm1, \pm i, \pm j , \pm k \}$. There are five irreducible representations
of Gal$(L/\mathbb Q)$. We denote the non-trivial $1$-dimensional representations by $\chi_1, \chi_2 $
and $\chi_3$, and the $2$-dimensional representation by $\rho$. In this case, $\nu (\chi_1) =\nu (\chi_2)=
\nu(\chi_3)=1$ and $\nu (\rho)=-1$ are valid. If assuming $m_{\chi_1}=m_{\chi_2}=m_{\chi_3}=0$, then the
coefficients $M(\sigma )$ and $m(\sigma)$ are given as follows.
\begin{center}
\begin{tabular}{c|ccccc}
$\sigma$ & $1$ & $-1$ & $\pm i$ & $\pm j$ & $\pm k$  \\ \hline  
$M(\sigma)$ & $1/2 $ & $5/2$ & $-1/2$ & $-1/2$ & $-1/2$  \\ 
$m(\sigma)$ & $2m_{\rho}$ & $-2m_{\rho}$ & $0$ & $0 $ & $0$  \\
 \end{tabular}
 \end{center}

 Fr\"{o}hlich \cite{F} 
proved that there exist infinitely many quaternion extensions $L/\mathbb Q$
 satisfying $m_{\rho} \ne 0$.
\end{ex}
To prove the theorem, we show the following proposition.
\begin{prop}\label{prop:main}
Let $K$ be a global field and $\rho$ a nontrivial irreducible Artin representation of 
Gal$(K^{\mathrm{sep}}/K)$.  
The following (i) and (ii) are equivalent.
\begin{enumerate}[(i)]
\item DRH(A) for $\rho$.
\item There exists a constant $c$ such that
\begin{equation*}
\sum_{N(\mathfrak p )\leq x } \frac{ \chi_{\rho} (\mathrm{Frob}_{\mathfrak p} ) }{N(\mathfrak p)^{\frac{1}{2} }}
=- \left( \frac{\nu (\rho) }{2} +m_{\rho} \right) \log{\log{x}} +c+o(1)\quad(x\to\infty).
\end{equation*}
\end{enumerate}
\end{prop}
\begin{proof}
By taking the logarithm of \eqref{limit}, DRH(A) for $\rho$ is equivalent to the existence of a constant $L$ such that
\begin{equation}\label{eq:log}
m_{\rho} \log{\log{x}}+ 
 \underbrace{  
\sum_{N(\mathfrak p)\leq x} \sum_{k=1}^{\infty} \frac{ \chi_{\rho}(\mathrm{Frob}^k_{\mathfrak p})}{kN(\mathfrak p)^{\frac{k}{2}}}
}_{(\ast)}
=L+o(1) \quad(x\to\infty).
\end{equation}
The double sum $(*)$ is decomposed into three subseries:
\begin{equation}\label{eq:ast}
(\ast)=\sum_{N(\mathfrak p)\leq x}  \frac{ \chi_{\rho}(\mathrm{Frob}_{\mathfrak p})}{N(\mathfrak p)^{\frac{1}{2}}}
+
 \sum_{N(\mathfrak p)\leq x}  \frac{ \chi_{\rho}(\mathrm{Frob}^2_{\mathfrak p})}{2N(\mathfrak p)}
+ 
\sum_{N(\mathfrak p)\leq x} \sum_{k=3}^{\infty}  \frac{ \chi_{\rho}(\mathrm{Frob}_{\mathfrak p}^k)}{kN(\mathfrak p)^{\frac{k}{2}}}.
\end{equation}
Since
\[
 \sum_{N(\mathfrak p)\leq x} \sum_{k=3}^{\infty}\  \vline \frac{ \chi_{\rho}(\mathrm{Frob}_{\mathfrak p}^k)}{kN(\mathfrak p)^{\frac{k}{2}}} \vline
< \frac{4\,  \mathrm{dim}\, \rho}{3} \zeta_K(3/2)  < \infty,
\]
the last term of (\ref{eq:ast}) is absolutely convergent. 
Denote the limit by
\begin{equation}\label{C1}
C_1=
\lim_{x\to\infty}
\sum_{N(\mathfrak p)\leq x} \sum_{k=3}^{\infty}  \frac{ \chi_{\rho}(\mathrm{Frob}_{\mathfrak p}^k)}{kN(\mathfrak p)^{\frac{k}{2}}}.
\end{equation}
Next the central term of (\ref{eq:ast}) is estimated by means of 
the Generalized Mertens' theorem (\cite[Theorem~5]{R}):
\[
\sum_{N(\mathfrak p) \leq x} \frac{\chi_{\tau}(\mathrm{Frob}_{\mathfrak p} )}{N(\mathfrak p)} = \mathrm{mult} ({\textbf 1};\tau)
\log{\log{x}} +c_\tau+o(1) \quad (x\to \infty)
\] 
with a constant $c_\tau$ depending on the representation $\tau$ of 
Gal$(K^{\mathrm{sep}}/K)$.
Thus the central term in \eqref{eq:ast} is expressed as
\begin{align}
\sum_{N(\mathfrak p)\leq x}  \frac{ \chi_{\rho}(\mathrm{Frob}^2_{\mathfrak p})}{2N(\mathfrak p)}
 & =\frac{1}{2} \sum_{ N(\mathfrak p)\leq x} \frac{ \chi_{\mathrm{sym}^2 \rho}(
\mathrm{Frob}_{\mathfrak p}) -\chi_{\wedge^2 \rho} (\mathrm{Frob}_{\mathfrak p})}{N(\mathfrak p)} \nonumber \\
& =\frac{\nu (\rho)}{2} \log{\log{x}} +C_2+o(1)\label{C2}
\end{align} 
with some constant $C_2$.

If assuming (i), then (ii) follows from (\ref{eq:log}) \eqref{C1} \eqref{C2} with $c=L-C_1-C_2$.
Conversely, if (ii) holds with some $c$, then \eqref{eq:log} holds with $L=c+C_1+C_2$.
\end{proof}

\begin{cor}\label{cor:quadratic}
Let $L$ be a quadratic extension of a global field $K$,
and $\chi$ the nontrivial character of Gal$(L/K)$.
The following (i) and (ii) are equivalent:
\begin{enumerate}[\textrm(i)]
\item DRH(A) for $L_K(s,\chi)$ holds.
\item There exists a Chebyshev bias toward the primes $\p$ such that
$\chi (\mathrm{Frob}_{\mathfrak p}) =-1$ with the asymptotic
\begin{equation}\label{cor:asympFrob}
\sum_{N(\mathfrak p )\leq x } \frac{ \chi (\mathrm{Frob}_{\mathfrak p} ) }{N(\mathfrak p)^{\frac{1}{2} }}
=- \left( \frac{1}{2} +m_{\chi} \right) \log{\log{x}} +c+o(1)\quad(x\to\infty)
\end{equation}
with some constant $c$.
\end{enumerate}
\end{cor}
Now we are ready to show Theorem~\ref{thm:main}.
\begin{proof}[Proof of Theorem~\ref{thm:main}]
First we prove that (i) implies (ii).
By the orthogonality relations of characters and Proposition~\ref{prop:main}, the following is obtained:
\begin{align*}
\lefteqn{[L:K] \pi_{\frac{1}{2}}(x\,  ;\sigma) }\\
& = [L:K] \sum_{\mathfrak p  \in S_{\sigma} \atop N(\mathfrak p)\leq x} \frac{1}{N(\mathfrak p)^{
\frac{1}{2} }} \\
& = |c_{\sigma}| \sum_{\mathfrak p \in S \atop N(\mathfrak p)\leq x} \sum_{\rho : \mathrm{irred.}} \frac{
\chi_{\rho} \left( \left( \frac{L/K}{\mathfrak p} \right)\right) \overline{\chi_{\rho} (\sigma)}}{N(\mathfrak p)^{\frac{1}{2}}} \\
& = |c_{\sigma} | \left\{ \pi_{\frac{1}{2},K} (x)-\sum_{\rho \ne {\textbf 1} }\overline{ \chi_{\rho} (\sigma)} \left( \frac{\nu (\rho)}{2} +m_{\rho} \right)
\log{\log{x}} \right\} +c +o(1)  \\
& = |c_{\sigma} | \{ \pi_{\frac{1}{2},K} (x) -(M(\sigma)+m(\sigma)) \log{\log{x}} \} +c+o(1).
\end{align*}                                                                                                                          
Next we show that (ii) implies (i). Let $\rho$ be a nontrivial irreducible representation of Gal$(L/K)$.  
Multiplying \eqref{cor:asympFrob} by $\chi_{\rho}(\sigma)$ and adding them for all $\sigma \in $Gal$(L/K)$,
we obtain
\begin{multline}\label{sum-sigma}
\pi_{\frac{1}{2},K} (x) \sum_{\sigma \in G} \chi_{\rho} (\sigma) -[L:K] \sum_{N(\mathfrak p) \leq x} \frac{ \chi_{\rho} (\mathrm{Frob}_{\mathfrak p} )}{
N(\mathfrak p)^{\frac{1}{2}} }\\
= \left( \sum_{\sigma \in G}
\chi_{\rho} (\sigma )M(\sigma)+\sum_{\sigma \in G} \chi_{\rho} (\sigma )m(\sigma )\right) \log{\log{x}} +c+o(1) \\
(x\to \infty)
\end{multline}
for some constant $c$. By the orthogonality relations of characters, we have $\sum_{\sigma \in G} \chi_{\rho} (\sigma)=0$
and
\begin{align*}
& \sum_{\sigma \in G} \chi_{\rho} (\sigma) M(\sigma)  =
\frac{1}{2} \sum_{\rho' \ne {\textbf 1} \atop \text{irred.}} \nu (\rho') \sum_{\sigma \in G} \chi_{\rho} (\sigma) \chi_{\rho'}(\sigma) 
 =\frac{|G|}{2} \nu (\rho) , \\
 & \sum_{\sigma \in G} \chi_{\rho} (\sigma )m (\sigma )=\sum_{\rho' \ne {\textbf 1} \atop \text{irred.} } m_{\rho'} \sum_{\sigma \in G}
 \chi_{\rho} (\sigma) \chi_{\rho'}(\sigma) =|G| m_{\rho}.
\end{align*}
From these results, (\ref{sum-sigma}) and Proposition~\ref{prop:main}, we reach DRH(A) for $\rho$.

The equivalence between (ii) and (iii)  follows from the fact that  $\sum_{\sigma \in G} M(\sigma) =\sum_{\sigma  \in G }m(\sigma)=0$.
\end{proof}

\section{Bias of primes in abelian extensions}
\subsection{General theory}
In this section, we consider biases of prime ideals in finite abelian extensions $L/K$ of global fields.
Put $G:=$Gal$(L/K)$. In the case of abelian extensions, the coefficient $M(\sigma)$ for $\sigma \in G $
in Theorem~\ref{thm:main} is given by the following:
\[
M(\sigma )=\begin{cases}
\frac{1}{2} ( |G/G^2|-1) & (\sigma \in G^2),\\
-\frac{1}{2} & (\sigma \not\in G^2).
\end{cases}
\]
Hence the following is obtained.
\begin{thm}\label{thm:abel}
Let $L/K$ be a finite abelian extension of global fields. 
The following (i),(ii) and (iii) are equivalent:                                     
\begin{enumerate}[\textrm(i)]                                                                        
\item DRH(A) holds for any nontrivial character of $G$.
\item
For any $\sigma \in G$,
there exists a constant $c$ such that
\begin{multline*}
\pi_{\frac{1}{2},K}(x)-[L:K] \pi_{\frac{1}{2}} (x;\sigma) \\
=\begin{cases}
\left( \frac{ |G/G^2|-1}{2} +m(\sigma) \right) \log{\log{x}} +c +o(1) & (\sigma \in G^2) \\
\left(-\frac{1}{2} +m(\sigma) \right) \log{\log{x}} +c+o(1) & (\sigma \not\in G^2)
\end{cases}\quad (x\to \infty),
\end{multline*}
where $m(\sigma ):= \sum_{\chi \ne {\textbf 1} } \chi (\sigma )m_{\chi}$
with $\chi$ running through all nontrivial characters of $G$.
\item
For any $\sigma, \tau \in G$,
there exists a constant $c$ such that
\begin{multline*}
\pi_{\frac{1}{2}} (x;\tau)-\pi_{\frac{1}{2}} (x;\sigma) \\
 = \begin{cases}
 \frac{1}{[L:K] } \left( \frac{ |G/G^2|}{2} +m(\sigma) -m(\tau) \right) \log{\log{x}} +c+o(1) \\
 \hfill   (\sigma \in G^2, \tau \not\in G^2) \\
 - \frac{1}{[L:K] } \left( \frac{ |G/G^2|}{2} -m(\sigma) +m(\tau) \right) \log{\log{x}} +c+o(1) \\
 \hfill   (\sigma \not\in G^2, \tau \in G^2) \\
  \frac{1}{[L:K] } \left( m(\sigma) -m(\tau) \right) \log{\log{x}} +c+o(1) \\
 \hfill   (\sigma, \tau \in G^2 \text{ or  } \sigma, \tau \not\in G^2) 
 \end{cases}\\
 (x\to \infty).
\end{multline*}
\end{enumerate}    
\end{thm}
\begin{rem}\label{remark}
From Theorem \ref{thm:abel},
if $G\ne G^2$, then there exists a Chebyshev bias against primes $\p$ such that $\left(\frac{L/K}{\mathfrak p} \right)\in G^2$.
See Example \ref{ex:H=G^2} for details.
Note that $G/G^2$ is the Galois group of the composite of all quadratic extensions over $K$ contained in $L$, and that $|G/G^2|=1$ is
equivalent to $|G|=[L:K]$ being odd.
\end{rem}
\begin{cor}\label{cor:H-difference}
Let $L/K$ be a finite abelian extension of global fields. 
The following (i) and (ii) are equivalent:                                     
\begin{enumerate}[\textrm(i)]                                                                        
\item DRH(A) holds for any nontrivial character of $G$.
\item For any subset $H$ of $G$, there exists a constant $c$ such that
\begin{multline}
\lefteqn{|H|\,  \pi_{\frac{1}{2}}  (x;G\, \setminus \, H )-|G \, \setminus \, H | \, \pi_{\frac{1}{2}} (x;H) } \\
= \frac{\log{\log{x}}}{2 |G^2| } (| G\, \setminus \, H | \, |H\cap G^2|-|G^2 \, \setminus \, H |\,  |H| ) \\
-\log{\log{x}} \sum_{\sigma \in H} m(\sigma)
+c+o(1)\quad(x\to\infty). 
\label{eq:G-H}
\end{multline}
\end{enumerate}    
\end{cor}

\begin{proof}
Assume (i), and it follows from Theorem~\ref{thm:abel} that
\begin{align*}
 \pi_{\frac{1}{2}} (x;H)   = & \sum_{\sigma \in H\cap G^2} \pi_{\frac{1}{2}} (x;\sigma) +\sum_{\sigma \in H \, \setminus \, G^2} \pi_{\frac{1}{2}} (x;\sigma) \\
 =& \frac{1}{|G|} \left( |H| \, \pi_{\frac{1}{2},K}(x)-\frac{\log{\log{x}}}{2} ( |H\cap G^2| \, |G/G^2|-|H|) \right. \\
 & \left. -\log{\log{x}} \sum_{\sigma \in H} m(\sigma) \right) +c+o(1)\quad(x\to\infty),
\end{align*}
and that
\begin{align*}
\lefteqn{\pi_{\frac{1}{2}} (x;G \, \setminus \, H) }\\
  = & \sum_{\sigma \in G^2 \, \setminus \, H } \pi_{\frac{1}{2}} (x;\sigma) +\sum_{\sigma \in G\, \setminus (H \cup G^2)} \pi_{\frac{1}{2}} (x;\sigma) \\
 =& \frac{1}{|G|} \left( |G \, \setminus \, H| \, \pi_{\frac{1}{2},K}(x)-\frac{\log{\log{x}}}{2} ( |G^2 \, \setminus \, H|\,  |G/G^2|-|G \, \setminus \, H|) \right. \\
 & \left. -\log{\log{x}} \sum_{\sigma \in G \, \setminus \, H} m(\sigma) \right) +c+o(1)\quad(x\to\infty).
\end{align*}
The assertion (ii) follows from these two identities
and $\sum_{\sigma \in G} m(\sigma )=0$.

Conversely, assume (ii).
We can show Theorem~\ref{thm:abel} (ii) by considering the subset $H=\{ \sigma \}$ for (\ref{eq:G-H}) for any $\sigma \in G$.
\end{proof}
\begin{ex}\label{ex:H=G^2}
Let $L/K $ be a finite abelian extension of global fields. 
Assume DRH(A) for  any nontrivial character of $G$.
Then there exists a constant $c$ such that
\begin{align}
\lefteqn{\pi_{\frac{1}{2}} (x;G\, \setminus \, G^2) -( |G/G^2|-1) \pi_{\frac{1}{2} }(x;G^2)} \nonumber \\
& =
\left( \frac{ |G/G^2|-1}{2} -\frac{1}{|G^2|} \sum_{\sigma \in G^2 }m(\sigma) \right) \log{\log{x}}
 +c +o(1)\label{eq:H=G^2}
\end{align}
as $x\to\infty$.
Let $P_2$ be the set of primes $\p$ 
of $K$
such that 
$\p \nmid D_{L/K}$ and 
$\left(\frac{L/K}{\mathfrak p} \right)\in G^2$ and 
$P_1$ be the complement of $P_2$
in the set of primes $\p$ of $K$ satisfying $\p \nmid D_{L/K}$.
A positive constant $\delta$ is given by Theorem \ref{theo:chebotarev} as
$$
\delta
=\f{|G\setminus G^2|}{|G^2|}
=
\lim_{x\to \infty}
\f{\#\{\p\in P_1\ |\ N(\p)\le x\}}{\#\{\p\in P_2\ |\ N(\p)\le x\}}.
$$
Then the left hand side of \eqref{eq:H=G^2} is equal to
$$
\sum_{\p\in P_1,\ N(\p)\le x} \frac{1}{\sqrt{N(\p)}}
-\delta \sum_{\p\in P_2,\ N(\p)\le x} \frac{1}{\sqrt{N(\p)}}.
$$
Thus we conclude that 
if $G\ne G^2$ and $m(\sigma)=0$ for any $\sigma$, 
there exists a Chebyshev bias against primes $\p$ such that $\left(\frac{L/K}{\mathfrak p} \right)\in G^2$.
\end{ex}
\begin{ex}\label{ex:difference}
Let $L/K $ be a quadratic extension of global fields with $G=\mathrm{Gal}(L/K) $ $=\langle \tau \rangle $.
The following (i) and (ii) are equivalent:
\begin{enumerate}[\textrm(i)]
\item DRH(A) for $L_K(s,\chi)$ holds for  the nontrivial character $\chi$ of $G$.
\item There exists a Chebyshev bias toward the primes $\p$ such that
$\left(\frac{L/K}{\mathfrak p} \right)=\tau$ with the asymptotic
\[
\pi_{\frac{1}{2}} (x;\tau)-\pi_{\frac{1}{2}} (x;1)=\left( \frac{1}{2} +m_{\chi} \right) \log{\log{x}} + c+o(1)
\qquad (x\to\infty)
\]
for some constant $c$.
\end{enumerate}
\end{ex}

\bigskip
In the following sections we will apply the above results to various finite quotients of 
the idele class group $C_K$ of $K$.
The class field theory asserts that there is a one-to-one correspondence between finite abelian extensions over $K$
and open subgroups of finite index of $C_K$, and there is an isomorphism called the Artin map
$$
\psi : C_K/N_{L/K} C_L \longrightarrow \mathrm{Gal}(L/K)
$$
for any finite abelian extension $L/K$.
Let $K_{\mathfrak p}$ be the completion of $K$ for a prime ideal $\mathfrak p$ of $K$,
and $\pi_{\mathfrak p}$ a prime element of $K_{\mathfrak p}$.
Let $[ \pi_{\mathfrak p} ] \in C_K$ be the image of a natural map $K_{\mathfrak p}^{\times} \to C_K/N_{L/K}C_L$.
For a prime ideal $\mathfrak p$ which does not divide the relative discriminant $D_{L/K}$ of $L/K$, we obtain
$\psi ([\pi_{\mathfrak p}])=\left(\frac{L/K}{\mathfrak p} \right)$, where $\left(\frac{L/K}{\mathfrak p} \right)$ is the Frobenius element in
Gal$(L/K)$. 
\subsection{Bias against splitting primes}

 We denote by $S_D$ the set of all prime ideals  of $K$ which split completely and let $S_N:=S \, \setminus \, S_D$.
For $x\in \R_{>0}$, the following is defined:
\begin{align*}
\pi_s (x;L/K)_D:=\sum_{\mathfrak p \in S_D \atop N(\mathfrak p)\leq x} \frac{1}{N(\mathfrak p)^s}, \\
\pi_s (x;L/K)_N:=\sum_{\mathfrak p \in S_N \atop N(\mathfrak p)\leq x} \frac{1}{N(\mathfrak p)^s} .
\end{align*}
Since a prime ideal  $\mathfrak p$ of $K$ with $\mathfrak p \nmid D_{L/K}$ decomposes completely in $L$ if and only if
$\left( \frac{L/K}{\mathfrak p} \right)=1$, it holds $\pi_s(x;L/K)_D=\pi_s(x;1)$ and hence 
the following is obtained from Theorem~\ref{thm:abel}:
\begin{multline*}
\pi_{\frac{1}{2},K}(x) -[L:K]\pi_{\frac{1}{2} }(x;L/K)_D \\
=\left( \frac{|G/G^2|-1}{2} +m(1) \right) \log{\log{x}} +c+o(1)\quad (x\to\infty)
\end{multline*}
for some constant $c$
under the assumption of DRH(A) in the case of number fields.
This asymptotic suggests that if $G\ne G^2$ and $m(1)=0$, then there exists a Chebyshev bias against
primes which split completely.
Indeed, the following theorem is obtained.
\begin{thm}\label{thm:decompose}
Let $L/K$ be a finite abelian extension of global fields.
We assume DRH(A) for any nontrivial characters of $G$ in the case of number fields.
The following asymptotic is valid for some constant $c$:
\begin{multline*}
\pi_{\frac{1}{2} }(x;
L/K)_N-([L:K]-1)\pi_{\frac{1}{2}}(x;
L/K)_D  \\
=
\left( \frac{ |G/G^2| -1}{2} + m(1) \right)
\log{\log{x}} + c+o(1)\quad
(x\to\infty).
\end{multline*}
\end{thm}
\begin{proof}
The assertion follows from Corollary~\ref{cor:H-difference} and the following  identities: 
\begin{align*}
\pi_s(x;L/K)_D & =\pi_s(x;1), \\
\pi_s(x;L/K)_N & =\pi_s(x;G\, \setminus \, \{1 \} )+
\sum_{\mathfrak p \in 
S,
\, \mathfrak p |D_{L/K} \atop N(\mathfrak p )\leq x} \frac{1}{N(\mathfrak p)^s}.
\end{align*}
\end{proof}
\begin{ex}\label{ex:decompose}
Let $L/K $ be a quadratic extension of global fields.
The following (i) and (ii) are equivalent:
\begin{enumerate}[\textrm(i)]
\item DRH(A) for $L_K(s,\chi)$ holds with $\chi$ the nontrivial character of $G=\mathrm{Gal}(L/K)$.
\item 

There exists a Chebyshev bias against splitting primes with the asymptotic:
\[
\pi_{\frac{1}{2}} (x;L)_N-\pi_{\frac{1}{2}} (x;L)_D=\left( \frac{1}{2} +m_{\chi} \right) \log{\log{x}} + c+o(1)
\quad(x\to\infty)
\]
for some constant $c$.
\end{enumerate}
\end{ex}
\subsection{Cyclotomic fields}
Let $L=\Q (\zeta_q)\ (q\geq 3)$ be a cyclotomic field. The Galois group $G:=\mathrm{Gal}(L/\Q)$ is isomorphic to the multiplicative group
$(\Z/q\Z)^{\times}$ and $a\in (\Z/q\Z)^{\times}$ corresponds to $\sigma_a= \left( \frac{L/\Q}{ (a) } \right) \in G$,
where $\left( \frac{L/\Q}{\cdot} \right)$ is the Artin symbol. The  group of characters of $G$  is isomorphic to the group of Dirichlet characters modulo $q$.
Let $t (q)$ be the number of distinct prime numbers dividing $q$, and
$$
t:= \begin{cases}
t (q) -1 & (2 ||q), \\
t (q) & (4||q \text{ or } 2\nmid q),\\
t (q)+1 & (8 |q).
\end{cases}
$$
Then $|G/G^2|= 2^{t }$ holds. For 
$x \in \R_{>0}$ and 
$a\in (\Z/q\Z)^{\times}$, let
$$
\pi_s(x;q,a) :=\sum_{p<x:\, \text{prime} \atop p\equiv a \!\!\!\!  \pmod{q} } \frac{1}{p^s}.
$$
Under the assumption of DRH(A) for any nontrivial character of $G$,
it follows from Theorem~\ref{thm:abel} that there exists a constant $c$ such that
\begin{multline}
\pi_{\frac{1}{2}, \Q}(x) -\varphi (q) \pi_{\frac{1}{2}} (x;q,a)  \\
=\begin{cases}
\left(\f{2^{t} -1}2+m(\sigma_a) \right) \log{\log{x}} +c+o(1)\\
 \hspace{5cm}  \text{ ($a$ is a quadratic residue modulo $q$)} \\
\left(-\frac{1}{2} +m(\sigma_a) \right) \log{\log{x}} +c+o(1) \quad  \text{(otherwise)}
\end{cases} \\
(x\to\infty),\label{eq:cyclo}
\end{multline}
where $\varphi (q)$ is Euler's totient function.
It is expected \cite{C} that $L(\frac{1}{2} ,\chi) \ne 0$ for any Dirichlet character $\chi$, that is $m(\sigma)=0$
for any $\sigma\in G$.
Under this assumption we have the following.
\begin{cor}\label{cor:quadraticresidue}
Assume DRH(A) for $L(s,\chi)$ and that $L(\f12,\chi)\ne0$ for Dirichlet characters $\chi$.
If $m(\sigma_a)=0$ for any $a\in (\Z/q\Z)^{\times}$, there exists a Chebyshev bias
against primes which are quadratic residues modulo $q$.
\end{cor}
Moreover the size of the bias grows with an increase in the number of distinct prime divisors of $q$.
Indeed under the assumption of $m(\sigma_a)=m(\sigma_b)=0$, it follows from (\ref{eq:cyclo}) that
\begin{equation}\label{eq:difference}
\pi_{\frac{1}{2}} (x;q,b) -\pi_{\frac{1}{2}}(x;q,a)=\frac{2^{t-1}}{\varphi (q) }\log{\log{x}} +c+o(1)
\quad (x\to \infty)
\end{equation}
for some constant $c$ with $a$ and $b$ being a quadratic residue and a non-residue modulo $q$, respectively.
We also have
$$
\pi_{\frac{1}{2}} (x;q,b) -\pi_{\frac{1}{2}}(x;q,a)= c+o(1)\quad (x\to \infty),
$$
when both $a$ and $b$ are quadratic residues or non-residues modulo $q$
under the assumption of $m(\sigma_a)=m(\sigma_b)=0$.

\begin{cor}
Assume DRH(A) for $L(s,\chi)$ and that $L(\f12,\chi)\ne0$ for Dirichlet characters $\chi$.
If $m(\sigma_a)=0$ for any $a\in (\Z/q\Z)^{\times}$, there exist no biases between
any pair of quadratic residues,
and between any pair of quadratic non-residues.
\end{cor}

\begin{ex}[Chebyshev's original case $q=4$]\label{original}
The group $(\Z/4\Z)^\times$ has two elements $1$ and $3\pmod 4$,
which are a quadratic residue and a non-residue, respectively. 
From Proposition \ref{prop:main}, DRH(A) for $L_\Q(s,\chi)$ with $\chi$
the nontrivial character of $(\Z/4\Z)^\times$ is equivalent to the asymptotic
$$
\pi_{\f12}(x;\,4,\,3)-\pi_{\f12}(x;\,4,\,1)
=\f12\log\log x +c+o(1)\qquad (x\to\infty).
$$
This elucidates the original question of Chebyshev.
\end{ex}

\begin{ex}[$q=8$]
The group $(\Z/8\Z)^\times$ has four elements $1$, $3$, $5$ and $7\pmod 8$,
all of which except $1$ are quadratic non-residues. 
Thus DRH(A) for $L_\Q(s,\chi)$ for all nontrivial characters $\chi$ of $(\Z/8\Z)^\times$ is equivalent to the
following asymptotics:
for $j=3,\,5,\,7$
\begin{align*}
\pi_{\f12}(x;\,8,\,j)-\pi_{\f12}(x;\,8,\,1)
&=\f12\log\log x +c+o(1)\qquad (x\to\infty),
\end{align*}
and for any pairs of $j$, $k\in\{3,\,5,\,7\}$
$$
\pi_{\f12}(x;\,8,\,j)-\pi_{\f12}(x;\,8,\,k)=c+o(1)\qquad (x\to\infty).
$$
\end{ex}

\subsection{Cyclotomic function fields}
Let $K=\mathbb F_q(T)$ and $L=K(\Lambda_M)\, (M\in \F_q[T], M\ne 0)$ be a cyclotomic function field introduced by Carlitz \cite{Ca}. 
The set $\Lambda_M$ is given by $\Lambda_M := \{ u \in \overline{K} \ |\ u^M =0 \}$
where $u^M:=M(\varphi +\mu )(u)$,  and $\varphi $ and $\mu$ the endomorphisms of $\F_q$-module $\overline{K}$ defined by
$\varphi (u) :=u^q$ and $\mu (u) :=Tu$.
Hayes \cite{H}  developed Carlitz' ideas to an explicit class field theory for algebraic function fields.
The Galois group $G=\mathrm{Gal}(L/K)$ is isomorphic to  the multiplicative group $(\F_q[T]/M\F_q[T])^{\times}$.
A residue class $A \in (\F_q[T]/M\F_q[T])^{\times}$ corresponds to $\sigma_A := \left( \frac{L/K}{(A) } \right) \in G$,
where $\left( \frac{L/K}{\cdot} \right)$ is the Artin symbol. For a positive integer $n$ and $A\in (\F_q[T]/M\F_q[T])^{\times}$,
use the following:
\begin{align*}
\pi_{s,K}(q^n )& :=\sum_{\deg{P} \leq n} \frac{1}{N(P)^s},\\
\pi_s (q^n;M,A)&:=\sum_{\deg{P} \leq n \atop P\equiv A \!\!\!\! \pmod{M}} \frac{1}{N(P)^s},
\end{align*}
where $P$ runs through all monic irreducible
polynomials
 in $\F_q[T]$.
From Theorem~\ref{thm:abel} there exists a constant $c$ such that
\begin{multline*}
\pi_{\frac{1}{2}, K}(q^n) -\Phi (M) \pi_{\frac{1}{2}} (q^n;M,A) \\
=\begin{cases}
\left( \f{2^t -1}2+m(\sigma_A) \right) \log{n}+c+o(1)   \\
\hspace{5cm}  \text{ ($A$ is a quadratic residue modulo $M$)} \\
\left(-\frac{1}{2} +m(\sigma_A) \right) \log{n}  +c+o(1)  \quad (n\to \infty) \quad  \text{(otherwise)}
\end{cases} \\
(n\to\infty),
\end{multline*}
where $\Phi (M):= |(\F_q[T]/M \F_q[T])^{\times}|$ and $t:=\dim_{\F_2}(G/G^2)$.
Note that if $M$ is a product of different $r$ irreducible polynomials, then
$t$ is given by $t=\begin{cases}
1 & \text{($q$ is a power of $2$)} \\
2^r & \text{(otherwise)}
\end{cases}$.

\begin{cor}
If $m(\sigma_A)=0$ for any $A \in (\F_q[T]/M\F_q[T])^{\times}$, there exists a Chebyshev bias
against irreducible polynomials which are quadratic residues modulo $M$.
\end{cor}

\begin{ex}
Put $q=2$ and $M=T^2$. 
The quotient ring $\F_2[T]/(T^2)$ has two invertible elements $1$ and $T+1$.
Since $m(\sigma_1)=m(\sigma_{T+1})=0$, there exists a Chebyshev bias 
toward irreducible polynomials congruent to the quadratic non-residue $T+1$.
If we denote an irreducible polynomial by $h=\sum_{j=0}^n a_j T^j$ $(a_j\in\F_2)$,
there exists a Chebyshev bias toward those with $a_1=1$.
\end{ex}

\subsection{Bias against principal ideals
and divisors
}
Let $K$ be a global field. In this section, we apply Theorem~\ref{thm:abel} for finite
abelian extensions corresponding to the ideal class group and the divisor class group of degree $0$ which are
quotients of the idele class group of $K$. 
Let $I_K$ be the group of fractional ideals of $K$, $P_K$ its
subgroup of principal ideals, $\mathrm{Cl}_K :=I_K/P_K$ the ideal class group and $h_k:=|\mathrm{Cl}_K|$ the class number.
We denote by $\widetilde{K}$ the Hilbert class field of $K$, hence we have  $\mathrm{Cl}_K \simeq \mathrm{Gal} (\widetilde{K}/K)$.
An ideal class $[\mathfrak a] \in \mathrm{Cl}_K$ corresponds to $\sigma_{\mathfrak a}:=\left( \frac{\widetilde{K}/K}{\mathfrak a} \right) \in \mathrm{Gal}
(\widetilde{K}/K)$. 

We assume DRH(A) below for the case that $K$ is a number field.
From Theorem~\ref{thm:abel} the following is obtained:
\begin{align*}
& \sum_{N(\mathfrak p)\leq x} \frac{1}{\sqrt{N (\mathfrak p )} }-h_K \sum_{\mathfrak p\in P_K \atop N(\mathfrak p) \leq x} 
\frac{1}{\sqrt{N(\mathfrak p)}} \\
=& \left(\frac{ |{\mathrm Cl}_K/\mathrm{Cl}_K^2|-1}{2}+m(1) \right) \log{\log{x}} + c+o(1)
\quad(x\to\infty) 
\end{align*}
for some constant $c$,
where $\p$ runs through prime ideals of $K$. 
Furthermore, the following holds:
\begin{multline*}
\sum_{\mathfrak p  \not\in P_K \atop N(\mathfrak p ) \leq x} \frac{1}{\sqrt{N(\mathfrak p)}} -(h_K-1) \sum_{\mathfrak p\in P_K \atop
  N(\mathfrak p)\leq x } \frac{1}{\sqrt{N(\mathfrak p)}} \\
=
\left(\frac{ |\mathrm{Cl}_K/\mathrm{Cl}_K^2| -1}{2} +m(1) \right)
\log{\log{x}}  +c+o(1)\quad
(x \to \infty)
\end{multline*}
for some constant $c$. 

Hence we have the following.
\begin{cor}\label{cor:principal}
Assume DRH(A) for $L_K(s,\sigma_{\mathfrak a})$ and that $m(\sigma_{\mathfrak a})=0$
for any $[\mathfrak a] \in \mathrm{Cl}_K$.
If $|\mathrm{Cl}_K|$ is even, then
there exists a Chebyshev bias against principal ideals in the whole set of prime ideals of $K$. 
\end{cor}
\begin{ex}\label{ex:h-quadratic}
Let $K=\Q(\sqrt{D})$  be a quadratic field with the discriminant $D$. Let $P_K^+$
be the group of totally positive principal ideals and $\mathrm{Cl}_K^+:=I_K/P_K^+$  the narrow ideal class group of $K$ and
$h_K^+:=|\mathrm{Cl}_K^+|$. We apply Theorem~\ref{thm:abel} for the finite abelian extension $L/K$
corresponding to $\mathrm{Cl}_K^+$, hence $\mathrm{Cl}_K^+ \simeq
\mathrm{Gal} (L/K)$ holds. From the genus theory we obtain 
\begin{align*}
& \sum_{N(\mathfrak p)\leq x} \frac{1}{\sqrt{N (\mathfrak p )} }-h_K^+ \sum_{\mathfrak p\in P_K^+ \atop N(\mathfrak p )\leq x} 
\frac{1}{\sqrt{N(\mathfrak p)}} \\
=& \left(\frac{ 2^{t(D) -1} -1}{2}+m(1) \right) \log{\log{x}} + c+o(1)
\quad (x\to \infty)
\end{align*}
for some constant $c$, where $t(D)$ is the number of distinct prime numbers dividing $D$.
\end{ex}
\begin{ex}\label{ex:D0}
Let $K$ be an algebraic function field of one variable over $\F_q$.
Let $D_K^0$ be the group of divisors of degree $0$, $H_K$ its subgroup of principal divisors, ${\mathfrak Cl}_K:=D_K^0/H_K$
 the divisor class group of degree $0$ and ${\mathfrak h}_K:=|{\mathfrak Cl}_K|$.
We apply Theorem~$\ref{thm:abel}$ for the finite abelian extension $L/K$ corresponding to ${\mathfrak Cl}_K$, hence
 ${\mathfrak Cl}_K \simeq \mathrm{Gal}(L/K)$ holds. 
The following is obtained.
\begin{align*}
& \sum_{\mathrm{deg} P\leq n} \frac{1}{\sqrt{N (P )} }-{\mathfrak h}_K \sum_{ P\in H_K \atop \mathrm{deg}P \leq n} 
\frac{1}{\sqrt{N(P)}} \\
=& \left(\frac{  | {\mathfrak Cl}_K/{\mathfrak Cl}_K^2 |-1}{2}+m(1) \right) \log{n} + c+o(1)
\quad (n\to \infty)
\end{align*}
for some constant $c$, where the sums are taken over monic irreducible polynomials in $\F_q[T]$.
\end{ex}

\section{Perspectives}
The method of the proof of Theorem \ref{thm:main} can be applied to more general $L$-functions.
This section introduces two examples as applications of our technique
as well as future prospects on Selberg zeta functions.

\subsection{Automorphic $L$-functions}
Let $\tau(n)$ be the Ramanujan's function defined by
$$
\Delta(z)=q\prod_{k=1}^\infty (1-q^k)^{24}
=\sum_{n=1}^\infty \tau(n)q^n.
$$
Kurokawa and the second author obtain a bias on its signature as follows:
\begin{thm}[Koyama-Kurokawa \cite{KK}]\label{tau}
Assume DRH(A) for $L(s+\f{11}2,\Delta)$. Then the sequence
$a(p)=\tau(p)p^{-11/2}$ has a Chebyshev bias to being positive.
More precisely, it holds that
$$
\sum_{\genfrac{}{}{0pt}{1}{p\le x}{\text{prime}}}\f{\tau(p)}{p^6}=\f12\log\log x+c+o(1)
\quad(x\to\infty)
$$
with some constant $c$.
\end{thm}
Here Ramanujan's $L$-function is defined as
$$
L(s,\Delta)=\sum_{n=1}^\infty \f{\tau(n)}{n^s}\quad \left(\Re(s)>\frac{13}{2}\right). 
$$
Theorem \ref{tau} suggests that the Satake parameters 
$\theta(p)\in[0,\pi]$ have a bias to being in $[0,\pi/2]$, where we define $\theta(p)$ as
$\tau(p)=2p^{\f{11}2}\cos(\theta(p))$.
In general, biases of the Satake parameters for automorphic forms on $GL(n)$ may be obtained. 
For example, according to \cite{Koyama}, under DRH(A) for the symmetric square $L$-function of $L(s,\Delta)$,
which is an $L$-function for $GL(3)$,
a bias to being \textit{negative} was found for a sequence over squares of primes. Namely,
$$
\sum\limits_{p\le x}\f{\tau(p^2)}{p^{\f{23}2}}
=-\f12\log\log x+c+o(1)\quad(x\to\infty)
$$
with some constant $c$.

Sarnak \cite{S} also reached a similar prediction under the assumption of
the Generalized Riemann Hypothesis as well as the Grand Simplicity Hypothesis for $L(s,\Delta)$,
which asserts linear independence over $\Q$ of the imaginary parts of all nontrivial zeros of $L(s,\Delta)$
in the upper half plane.
He has pointed out that the sum
$$
S(x)=\sum_{\genfrac{}{}{0pt}{1}{p\le x}{\text{prime}}}\f{\tau(p)}{p^\f{11}2}
$$
has a bias to the positive side, in the sense that the mean of the measure $\mu$ defined by
$$
\f1{\log X}\int_2^Xf\l(\f{\log x}{\sqrt x}S(x)\r)\f{dx}x\to\int_\R f(x)d_\mu(x)\quad 
(x\to\infty)
$$
for $f\in C(\R)$ is equal to 1.
In the proof he closely examines the logarithmic derivative of $L(s,\Delta)$
to find that the second term in its expansion is the cause of the above bias.
While our discussion above deals with the logarithm instead of its derivative,
we have also reached the point that the bias derives from the second term in the expansion.
Although our conclusion has a common cause of the bias with what is given in \cite{S},
our proof is straightforward enough to simplify the proofs.

\subsection{$L$-functions of elliptic curves}
Let $E$ be an elliptic curve over $K$. 
For any finite place $v$ of $K$, we put
$$
a_v=a_v(E)=q_v+1-\#E_v^{\mathrm{ns}}(k_v),
$$
where 
$q_v$ is the cardinal of the residue field $k_v$, and 
$\#E_v^{\mathrm{ns}}(k_v)$ is the number of $k_v$-rational points on the nonsingular locus 
of the reduction $E_v$ of $E$ at $v$.
Defining
$\theta_v\in[0,\pi]$ as $a_v=2q_v^{\frac{1}2}\cos(\theta_v)$,
the $L$-function
$$
L(s,M_E)=\prod_{v:\text{ good}}(1-2\cos(\theta_v)q_v^{-s}+q_v^{-2s})^{-1}
\prod_{v:\text{ bad}}(1-a_vq_v^{-s})^{-1}
$$
satisfies a functional equation for $s\leftrightarrow 1-s$.

By fixing $\ell$ and $K$ with $\ell\ne\mathrm{char}(K)$, we are equipped with a representation 
$$
\rho_E:\ \mathrm{Gal}(K^{\mathrm{sep}}/K)\to \mathrm{Aut}(T_\ell(E)\otimes\Q_\ell),
$$
where $T_\ell(E)=\varprojlim\limits_n E[\ell^n]$ is the $\ell$-adic Tate module of $E/K$.
It holds that $L(s,M_E)=L_K(s,\rho_E)$.
If $\rho_E$ is irreducible and DRH(A) is true for $L_K(s,\rho_E)$,
a bias of the sequence ${a_v}/{\sqrt{q_v}}$ is obtained.
Ulmer \cite{U} proves that  $\rho_E$ is irreducible, when $\mathrm{char}(K)>0$ and $E$ is nonconstant.
By using this result, Kaneko and the second author obtain the following theorem.

\begin{thm}[Kaneko-Koyama \cite{KK2}
]\label{th:elliptic}
Assume $\mathrm{char}(K)>0$ and that $E$ is a non-constant elliptic curve over $K$.
If $\mathrm{rank}(K)>0$, then the sequence ${a_v}/{\sqrt{q_v}}$ has a Chebyshev bias to being negative.
More precisely, the following holds:
\begin{equation}\label{biascor1}
\sum\limits_{q_v\le x}\frac{a_v}{q_v}
= C\log\log x+O(1)
\quad(x\to\infty),
\end{equation}
where $C=\frac12-\mathrm{rank}(K)$.
\end{thm}

The proof uses the DRH(A) proved for $\mathrm{char}(K)>0$ 
and a part of the Birch-Swinnerton-Dyer Conjecture proved by Ulmer in \cite{U}
for $\mathrm{char}(K)>0$.

\subsection{Selberg zeta functions}
In positive characteristic cases, the Selberg zeta functions for congruence subgroups of
$PGL(2,\F_q[T])$ attached to nontrivial representations satisfy both DRH(A) and (B) (Koyama-Suzuki \cite{KS}).
This allows us to consider biases of closed paths in the Ramanujan diagrams (i.e. infinite Ramanujan graphs) 
corresponding to the congruence subgroups.

When a Ramanujan diagram $X$ is a cover of $X_0$ of finite degree,
a Chebyshev bias would exist in the set of closed paths in $X_0$ toward those coming from $X$.

For Riemann surfaces with characteristic zero, analogous phenomena are expected to emerge.
For this purpose it is needed to prove DRH(A) for Selberg zeta functions.
Kaneko-Koyama \cite{KK3} proved the convergence of the Euler products of Selberg zeta functions
for arithmetic subgroups of $PSL(2,\R)$ in the whole zero-free region.
Hence in the case of congruence subgroups of $PSL(2,\Z)$, 
the Euler product converges in $\f12<\Re(s)<1$ under the assumption of the Selberg $\f14$-conjecture.
If the convergence region is extended to the boundary, DRH(A) holds and we would be able to obtain
biases of prime geodesics on arithmetic surfaces toward those coming from covering spaces.

\section*{Acknowledgement} This work was supported by JSPS KAKENHI Grant Number JP21K03181.

\newpage
                                                                                           
\appendix{ {\bf Appendix: Numerical Evidence of the Chebyshev Biases } \\    
\begin{center}                                                                          
Miho Aoki, Shin-ya Koyama and Takamasa Yoshida        
\end{center}          
}                        
\renewcommand{\thefigure}{\Alph{section}}
\setcounter{figure}{0}

\maketitle
\begin{abstract}
We give numerical evidences for the facts which were proved in the joint paper of
the first and the second authors \cite{AK} under the assumption of the Deep Riemann Hypothesis (DRH).

\end{abstract}
\section{Introduction}
In this appendix we show numerical evidences for the Deep Riemann Hypothesis (DRH),
by verifying the asymptotics in \cite{AK} numerically.
We fulfilled the calculations for $p\le10^{10}$.

We used the software MATLAB R2021b (Mathworks, Natick, MA) 
and illustrated the results by plotting the points in the following manner for the purpose of suppressing the load to the computer system:
\begin{itemize}
\item Showed all values for $p\le 10^5$.
\item Showed every ten values for $10^5<p\le 10^{10}$.
\end{itemize}

\section{Chebyshev's original case}
Let $q\ge 3$ be an integer. 
Denote by $\pi_{\frac{1}{2}}(x;q,a)$ the sum of the reciprocals of the square roots of primes $p\le x$
satisfying $p\equiv a\pmod q$.
Under the assumption of the Deep Riemann Hypothesis (DRH)
and non-vanishing of the Dirichlet $L$-function $L(s,\chi)$ at $s=1/2$ for any 
nontrivial Dirichlet character $\chi$ modulo $q$,
it follows from \eqref{eq:difference} in \cite{AK} that
\begin{equation}\label{eq:differenceA}
\pi_{\frac{1}{2}} (x;q,b) -\pi_{\frac{1}{2}}(x;q,a)\sim\frac{2^{t-1}}{\varphi (q) }\log{\log{x}}
\quad (x\to \infty)
\end{equation}
with $a$ and $b$ being a quadratic residue and a non-residue modulo $q$, respectively,
where the integer $t$ is defined by $|G/G^2|= 2^{t }$
with $G=\mathrm{Gal}(\Q(\zeta_q)/\Q)$.

We restore Chebyshev's original case by choosing $q=4$.
Denote by $\chi$ the nontrivial character of $G=\mathrm{Gal}(\Q(\sqrt{-1})/\Q)\cong \Z/2\Z$.
Then \eqref{eq:differenceA} tells that DRH for $L(s,\chi)$ implies
\begin{equation}\label{eq:differenceB}
\pi_{\frac{1}{2}} (x;4,3) -\pi_{\frac{1}{2}}(x;4,1)\sim\frac{1}{2}\log{\log{x}}
\quad (x\to \infty).
\end{equation}
The graph of the left hand side of \eqref{eq:differenceB}, which we put as $S(x;4,1)$, is drawn in blue in Figure \ref{fig2}
with the horizontal axis being with the logarithmic scale,
while the right hand side of \eqref{eq:differenceB} is shown in brown.

\begin{figure}[H]
\begin{center} 
\includegraphics[width=7cm]{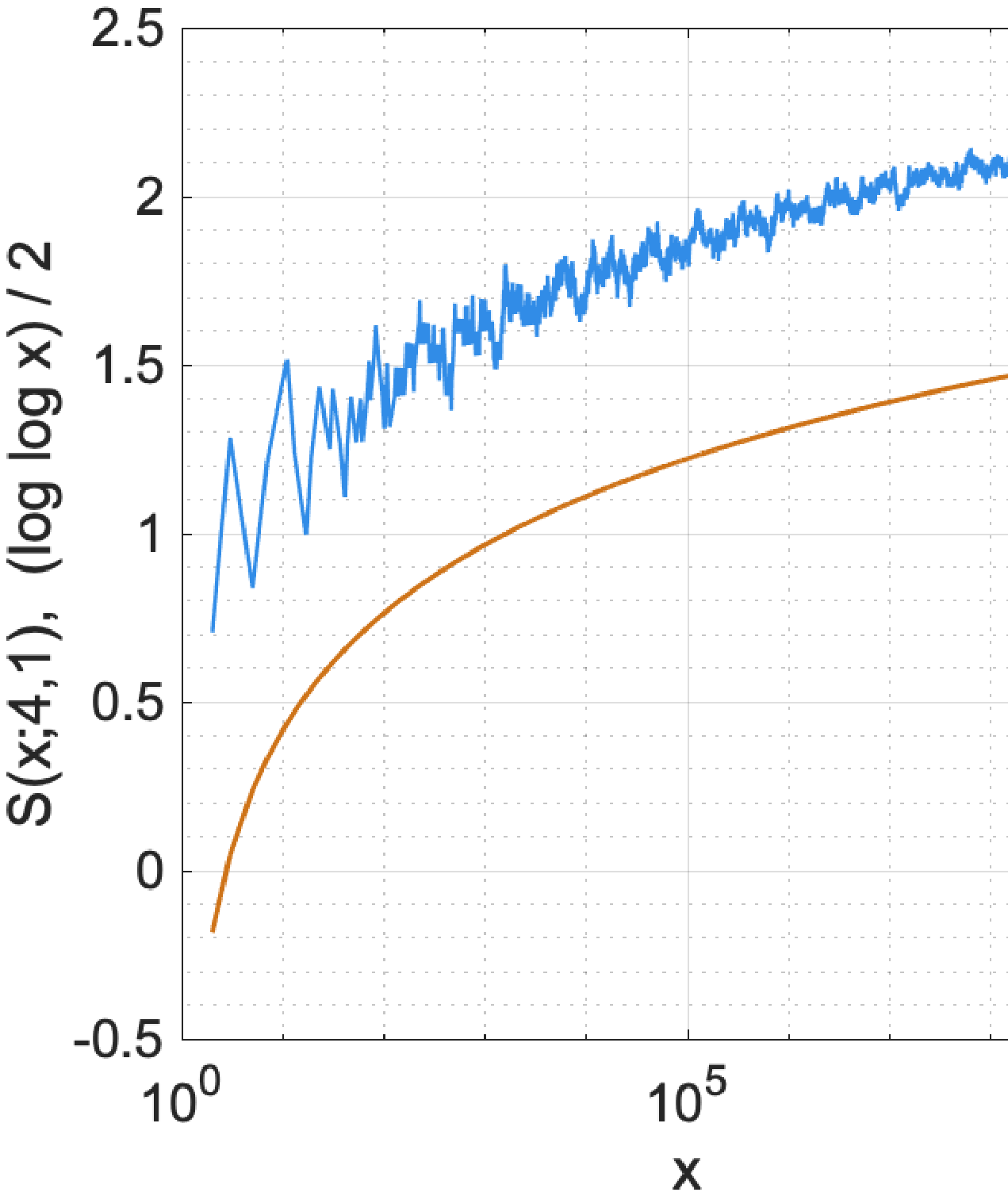}
\end{center}
\caption{Chebyshev's original case}
\label{fig2}
\end{figure}

\section{Other cyclotomic fields}
For abelian extension of global fields $L/K$, we put
\[
\pi_{s,K}(x) := \sum_{N (\mathfrak p)  \leq x} \frac{1}{N(\mathfrak p)^s}.
\]
Assume DRH and $L(1/2,\chi)\ne0$ for any nontrivial Dirichlet character modulo $q$.
It follows from \eqref{eq:cyclo} that under such assumptions
\begin{multline}
\pi_{\frac{1}{2},\Q}(x) -\varphi (q) \pi_{\frac{1}{2}} (x;q,a)  \\
\sim\begin{cases}
\f{2^{t} -1}2 \log{\log{x}}&\text{ ($a$ is a quadratic residue modulo $q$)} \\
-\frac{1}{2}  \log{\log{x}}&\text{(otherwise)}
\end{cases} 
(x\to\infty).\label{C}
\end{multline}

Let $q=60$. 
We have $\varphi (60)=16$ and $t=3$.
There exist two quadratic residues $a=1,49$ and fourteen non-residues $a=7,11,13,17,19,23,29,31,37,$ $41,43,47,53,59$.
Denoting the left hand side of \eqref{C} by $S(x;q,a)$, we show the graphs of 
$S(x;60,1)$ and $S(x;60,7)$ in Figure \ref{fig3} in blue.
The right hand side of \eqref{C} is shown in brown in each graph.
\begin{figure}[H]
\begin{center} 
\includegraphics[width=7cm]{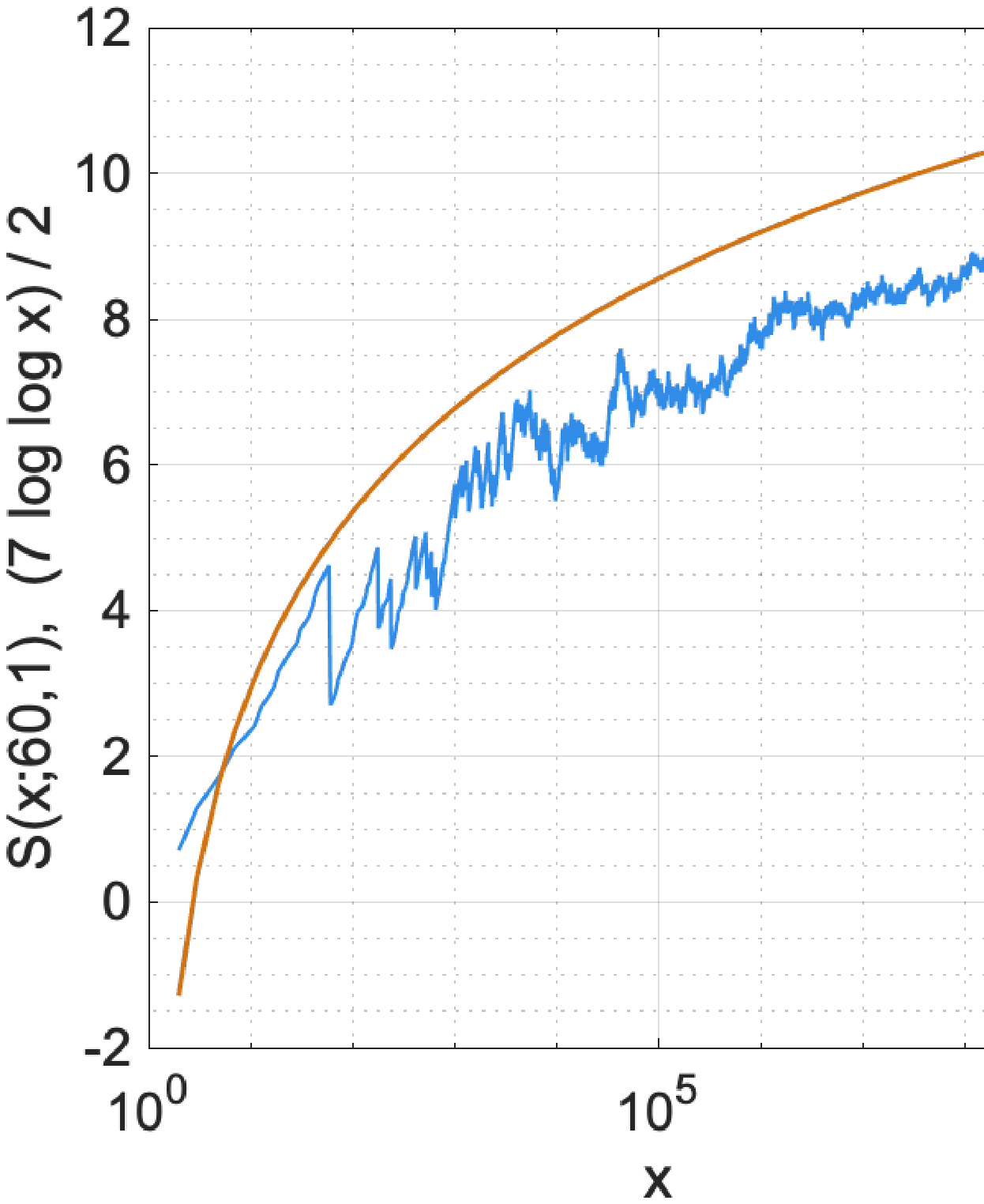} \hspace{0.2cm}
\includegraphics[width=7cm]{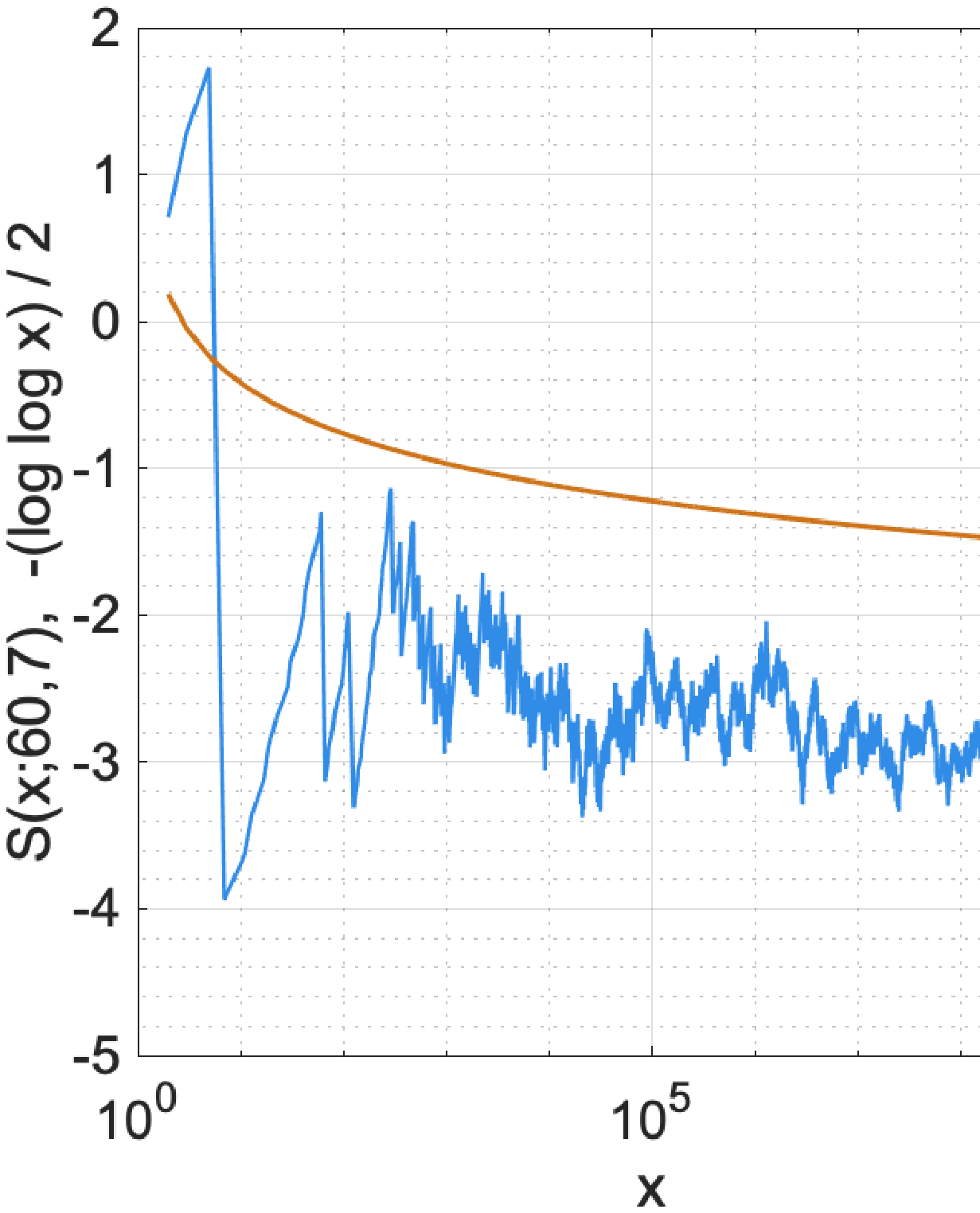}
\end{center}
\caption{Quadratic residue and non-residue modulo 60}
\label{fig3}
\end{figure}

\section{Unbiased cases}   
Each $\sigma \in G:=$Gal$(L/K)$ is equipped with
$$
S_{\sigma} := \left\{ \mathfrak p \ \vline \ \mathfrak p \nmid D_{L/K}, \left( \frac{L/K}{\mathfrak p} \right) =\sigma \right\}
$$
and
$$
\pi_s(x;L/K,\sigma) := \sum_{\mathfrak p \in S_{\sigma } \atop N (\mathfrak p)  \leq x} \frac{1}{N(\mathfrak p)^s}.
$$
We choose $K=\Q$ and $L=\mathbb Q(\zeta_7+\zeta_7^{-1})$, which
is an intermediate field in $\mathbb Q(\zeta _7)/\mathbb Q$ with $[\mathbb Q (\zeta_7):L]=2$ and $[L:\mathbb Q]=3$.
The group $G:=$Gal$(L/\mathbb Q)$ is of order 3, and we put $G:=\{ 1, \sigma, \tau \}$ with
\begin{align*}
S_{1} & := \left\{ p  \ \vline \ p  \ne 7,\ \left( \frac{L/\mathbb Q}{(p)} \right) =1 \right\}= \{ p \ |\ p=\pm 1 \pmod{7} \}, \\
& \\
S_{\sigma} & := \left\{ p  \ \vline \ p  \ne 7,\ \left( \frac{L/\mathbb Q}{(p)} \right) =\sigma \right\}= \{ p \ |\ p=\pm 2 \pmod{7} \}, \\
& \\
S_{\tau} & := \left\{ p  \ \vline \ p  \ne 7,\ \left( \frac{L/\mathbb Q}{(p)} \right) =\tau \right\}= \{ p \ |\ p=\pm 3 \pmod{7} \}.
\end{align*}
Putting  
$
S(x;L/K,\sigma ) :=\pi_{\frac{1}{2},K}(x) -[L:K] \pi_{\frac{1}{2}} (x;L/K,\sigma ),
$
we will draw the graphs of the following three functions
\begin{align*}          
S(x;L/\mathbb Q ,1  )  
&:=\pi_{\frac{1}{2},\mathbb Q}(x) -3 \ \pi_{\frac{1}{2}} (x;L/\mathbb Q,1 ) =\pi_{\frac{1}{2},\mathbb Q}(x) -3 \sum_{p \leq x: \text{prime} \atop p \equiv 1,6 \!\!\!\! \pmod{7} } \frac{1}{\sqrt{p}}\\
S(x;L/\mathbb Q ,\sigma  )  
&:=\pi_{\frac{1}{2},\mathbb Q}(x) -3 \ \pi_{\frac{1}{2}} (x;L/\mathbb Q,\sigma ) =\pi_{\frac{1}{2},\mathbb Q}(x) -3 \sum_{p \leq x: \text{prime} \atop p \equiv 2,5 \!\!\!\! \pmod{7} } \frac{1}{\sqrt{p}} \\
S(x;L/\mathbb Q ,\tau  )  &:=\pi_{\frac{1}{2},\mathbb Q}(x) -3 \ \pi_{\frac{1}{2}} (x;L/\mathbb Q,\tau ) =\pi_{\frac{1}{2},\mathbb Q}(x) -3 \sum_{p \leq x: \text{prime} \atop p \equiv 3,4 \!\!\!\! \pmod{7} } \frac{1}{\sqrt{p}}
\end{align*}
under the assumptions of DRH and $L(1/2,\chi)\ne0$ for any nontrivial Dirichlet character $\chi$ modulo 7.
By Remark \ref{remark} in \cite{AK}, there would exist no biases when the order of the Galois group is odd.
We actually have almost flat graphs in Figure \ref{fig4}.
\begin{figure}[H]
\begin{center}
\includegraphics[width=7cm]{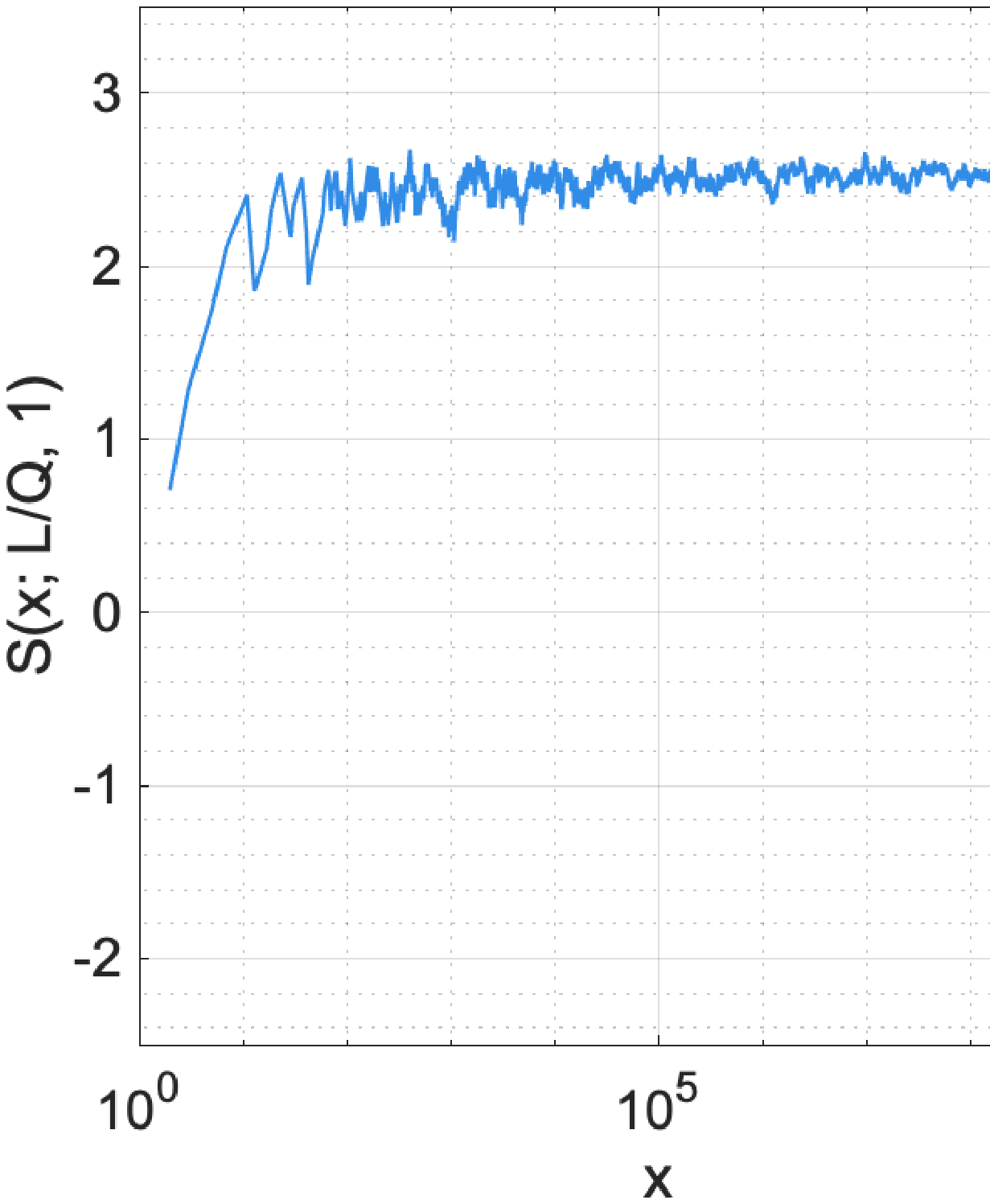}
\includegraphics[width=7cm]{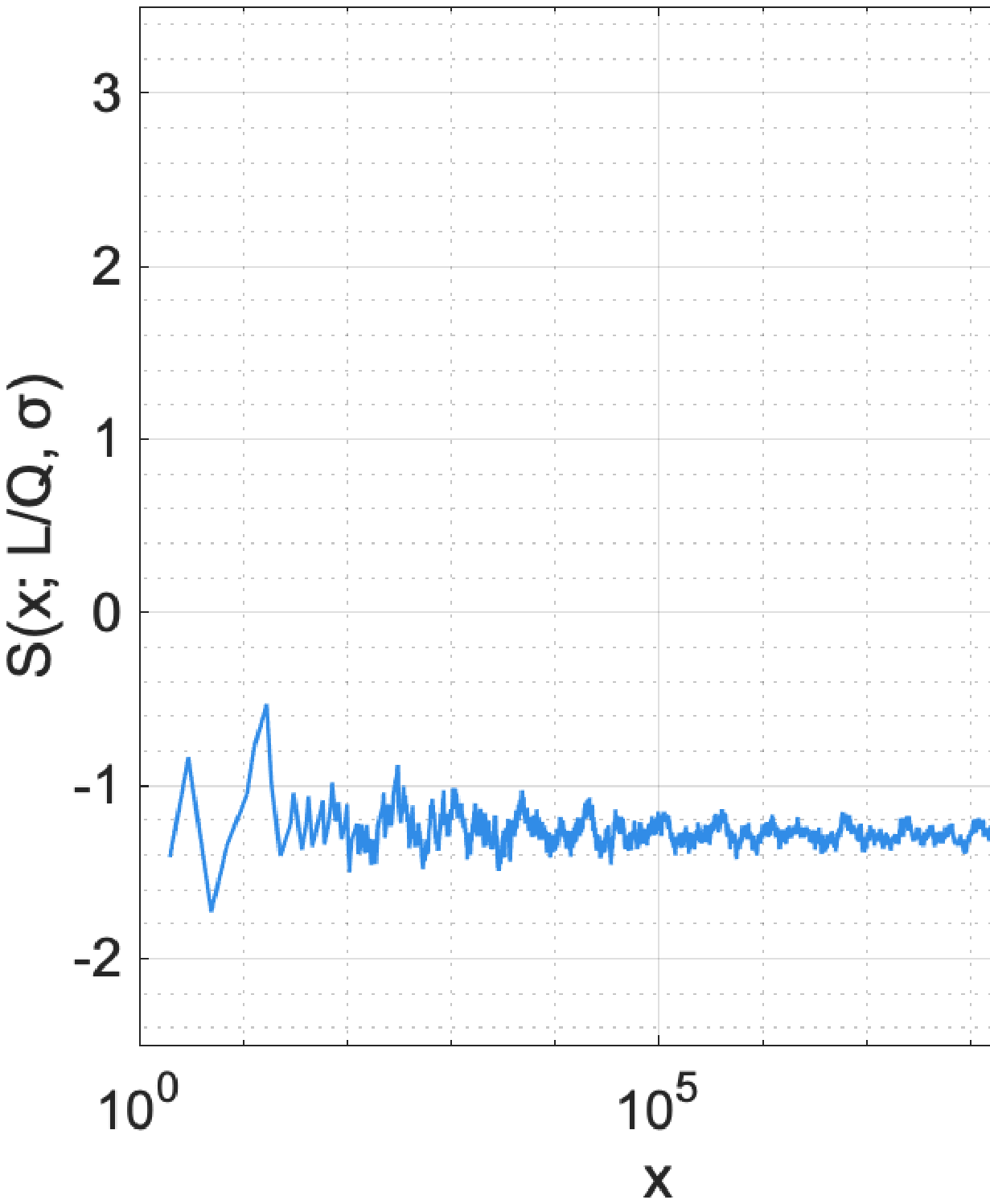}
\includegraphics[width=7cm]{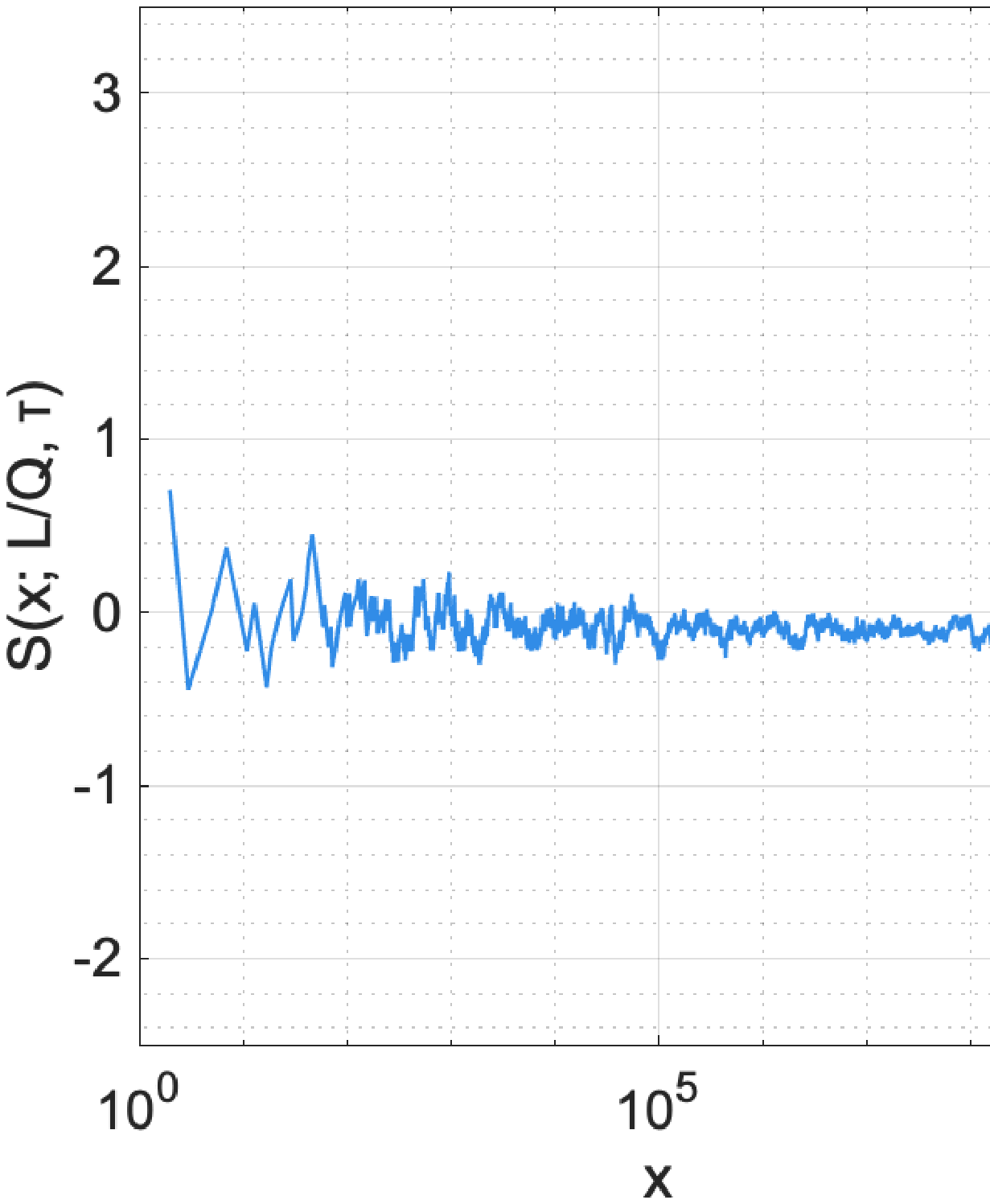}
\end{center}
\vskip -5mm
\caption{Unbiased cases}
\label{fig4}
\end{figure}


\begin{thebibliography}{99}
\bibitem{A}
H. Akatsuka:
The Euler product for the {R}iemann zeta-function in the critical strip.
Kodai. Math. J. \textbf{40} (2017) 79-101.
\bibitem{BSD}
B.J. Birch and H.P.F. Swinnerton-Dyer:
Notes on elliptic curves. {I\hspace{-.1mm}I}.
J. reine angew. Math., \textbf{218} (1965) 79-108.
\bibitem{Ca} 
L. Carlitz: On certain functions connected with polynomials in a Galois field, Duke Math. J. {\textbf 1} (1935), 137--168.
\bibitem{C}
S.D. Chowla: 
The Riemann Hypothesis and Hilbert's tenth problem. 
Gordon and Breach Science Publishers, New York, London, Paris (1965).
\bibitem{Co}
K. Conrad: Partial Euler products on the critical line.
Canad. J. Math. \textbf{57} (2005) 267-297. 
\bibitem{Fi} D.~Fiorilli: 
Highly biased prime number races.
Algebra Number Theory 8 (2014) 1733-1767.
\bibitem{FJ} D.~Fiorilli and  F.~Jouve: Unconditional Chebyshev biases in number fields,
J. \'{E}c. polytech. Math. 9 (2022), 671-679.
\bibitem{F} 
A.~Fr\"{o}hlich:
Artin root numbers and normal integral bases for quaternion fields, 
Invent.~Math.~17 (1972), 143--166. 
\bibitem{GL}       
S.R. Garcia and E.S. Lee:
Explicit Mertens' Theorems for Number Fields (preprint).
https://arxiv.org/abs/2006.03337
\bibitem{G} 
D. Goldfeld: Sur les produits partiels eul\'eriens attach\'es aux courbes elliptiques. 
C. R. Acad. Sci. Paris Ser. I Math. \textbf{294} (1982) 471-474. 
\bibitem{H} 
D. R. Hayes: Explicit class field theory for rational function fields, Trans. Amer. Math. Soc. \textbf{189} (1974), 77--91.
\bibitem{Kac}
J. Kaczorowski: On the distribution of primes (mod 4).
Analysis, \textbf{15} (1995) 159-171.
\bibitem{KK3}
I. Kaneko and S. Koyama:
Euler products of Selberg zeta functions in the critical strip.
Ramanujan J. (2022). https://doi.org/10.1007/s11139-022-00550-y
\bibitem{KK2}         
I. Kaneko and S. Koyama: Chebyshev's Bias for elliptic curves. (preprint) 
https://arxiv.org/abs/2206.05445
\bibitem{KKK2}
I. Kaneko, S. Koyama and N. Kurokawa:
Towards the deep Riemann hypothesis for $GL_n$. (preprint)
https://arxiv.org/abs/2206.02612
\bibitem{KKK}
T. Kimura, S. Koyama and N. Kurokawa: 
Euler products beyond the boundary.
Letters in Math. Phys. \textbf{104} (2014) 1-19.   
\bibitem{K}
H. Kornblum: \"Uber die Primfunktionen in einer arithmetischen Progression. 
Math. Zeit., {\textbf 5} (1919) 100-111.
\bibitem{Koyama}
S.~Koyama: 
Chebyshev's bias for Ramanujan's $\tau$-function via the Deep Riemann Hypothesis.
(Talk given at RIMS Conference on Automorphic Forms on January 29, 2022)
https://youtu.be/2IaNolUJjBE
\bibitem{KK}
S. Koyama and N. Kurokawa:
Chebyshev's Bias for Ramanujan's $\tau$-function via the Deep Riemann Hypothesis.
Proc. Japan Acad {\textbf 98A} (2022) 35-39.
\bibitem{KS}
S. Koyama and  F. Suzuki:
Euler products beyond the boundary for Selberg zeta functions.
Proc. Japan Acad. \textbf{90A} (2014) 101-106.
\bibitem{KM} 
W. Kuo and R. Murty: 
On a conjecture of Birch and Swinnerton-Dyer.
Canad. J. Math. \textbf {57} (2005) 328-337. 
\bibitem{Ku} 
N.~Kurokawa:
\newblock The pursuit of the Riemann Hypothesis (in Japanese). 
\newblock Gijutsu Hyoron-sha, 2012.
\bibitem{Li}
J. E. Littlewood:
Distribution des Nombres Premiers.
C. R. Acad. Sci. Paris \textbf{158} (1914) 1869-1872.
\bibitem{L}
L. Lafforgue:
Chtoucas de {D}rinfeld et correspondance de {L}anglands.
Invent. Math. \textbf{147} (2002) 1--241.
\bibitem{Le}
P. Lebacque:
Generalised Mertens and Brauer-Siegel theorems.
Acta Arith. \textbf{130} (2007) 333-350.
\bibitem{O}
S. Omar:
Non-vanishing of Dirichlet $L$-functions at the central point.
Algorithmic number theory, 443-453,
Lecture Notes in Comput. Sci., \textbf{5011}, Springer, Berlin, 2008.
\bibitem{R}
M. Rosen:
A generalization of Mertens' theorem.
J. Ramanujan Math. Soc. \textbf{14} (1999) 1-19.
\bibitem{R2} 
M. Rosen:
Number Theory in Function Fields.
Graduate Texts in Math. \textbf{210}, Springer-Verlag, New York, 2002.
\bibitem{RS}
M. Rubinstein and P. Sarnak:
Chebyshev's bias.
Exp. Math. \textbf {3} (1994) 173-197.
\bibitem{S}
P. Sarnak: 
Letter to Barry Mazur on ``Chebyshev's bias'' for $\tau(p)$.
November, 2007.
\bibitem{U}
D.~Ulmer:
Elliptic curves over function fields, Arithmetic of $L$-functions.
IAS/Park City Math. Ser., vol. 18: Amer. Math. Soc., Providence, RI, 2011, pp. 211-280.
\bibitem{W}
T. Willwacher:
Basic group and representation theory.
Lecture Notes, ETH Zurich, 2019.
\end{thebibliography}

\begin{thebibliography}{99}
\bibitem[A1]{AK}
M. Aoki and S. Koyama:
Chebyshev's Bias against splitting and principal primes in global fields.
(preprint)
\end{thebibliography}
\end{document}